\documentclass{alggeom}
\usepackage[mathscr]{eucal}
\usepackage{amsmath, palatino, mathpazo, amsfonts, mathrsfs}
\usepackage[all]{xy}
\usepackage{url}

\newtheorem{theorem}{Theorem}[section]
\newtheorem{lemma}[theorem]{Lemma}
\newtheorem{proposition}[theorem]{Proposition}
\newtheorem{corollary}[theorem]{Corollary}
\theoremstyle{remark}
\newtheorem{remark}[theorem]{Remark}

\newtheorem{definition}[theorem]{Definition}
\newtheorem{notations}[theorem]{Notations}
\newtheorem*{conventions}{Conventions}

\newcommand{\T}{\mathscr{F}}
\newcommand{\Mbar}{\overline{M}}

\newcommand{\n}{\pi}
\newcommand{\s}{\varepsilon}
\newcommand{\p}{\partial}
\newcommand{\f}{{\bf f}}

\newcommand{\e}{\mathbf{e}}

\newcommand{\m}{\nu}
\newcommand{\tb}{\bar{T}}
\newcommand{\bart}{\bar{t}}
\newcommand{\hatt}{\check}

\begin{document}

\title[Invariance of Quantum Rings I]
{Invariance of Quantum Rings under Ordinary Flops I:\\ Quantum corrections and reduction to local models}

\author[Y.-P.~Lee]{Yuan-Pin~Lee}
\email{yplee@math.utah.edu}
\address{Y.-P.~Lee: Department of Mathematics, University of Utah,
Salt Lake City, Utah 84112-0090, U.S.A.}

\author[H.-W.~Lin]{Hui-Wen~Lin}
\email{linhw@math.ntu.edu.tw}
\address{H.-W.~Lin: Department of Mathematics and Taida
Institute of Mathematical Sciences (TIMS),
National Taiwan University, Taipei 10617, Taiwan}

\author[C.-L.~Wang]{Chin-Lung~Wang}
\email{dragon@math.ntu.edu.tw}
\address{C.-L. Wang: Department of Mathematics,
Center for Advanced Studies in Theoretical Sciences (CASTS), and Taida
Institute of Mathematical Sciences (TIMS),
National Taiwan University, Taipei 10617, Taiwan}


\begin{abstract}
This is the first of a sequence of papers proving the quantum invariance
under \emph{ordinary flops} over an \emph{arbitrary smooth base}. 

In this first part, we determine the defect of the cup product 
under the canonical correspondence and show that it is corrected 
by the small quantum product attached to the extremal ray. 
We then perform various reductions to reduce the problem to the local models.

In Part II \cite{LLWp2}, we develop a \emph{quantum Leray--Hirsch theorem} and use it to show that the big quantum cohomology ring is
invariant under analytic continuations in the K\"ahler moduli space 
for \emph{ordinary flops of splitting type}.
In Part III \cite{LLQW}, we remove the splitting condition by developing a \emph{quantum splitting principle}, 
and hence solve the problem completely.
\end{abstract}

\maketitle

\tableofcontents


\numberwithin{equation}{section}
\setcounter{section}{-1}

\section{Introduction}

\subsection{Background review}

Two complex manifolds $X$ and $X'$ are \emph{$K$-equivalent},
denoted by $X =_K X'$, if there are proper birational morphisms
$(\phi, \phi'): Y \to X \times X'$ such that $\phi^* K_X = \phi'^*
K_{X'}$. Major examples come from \emph{birational minimal models}
in Mori theory and especially from \emph{birational Calabi--Yau
manifolds} in the mathematical study of string theory.
$K$-equivalent projective manifolds share the same Betti and Hodge
numbers. It has been conjectured that a \emph{canonical
correspondence} $T \in A(X \times X')$ exists which induces
isomorphisms of cohomology groups and preserves the
\emph{Poincar\'e pairing}. For a survey, see \cite{Wang2}.

However, simple examples show that the classical cup product is
generally not preserved under $\T$, and this leads to new
directions of study in higher dimensional birational geometry.
On the other hand, according to the philosophy of
\emph{crepant transformation conjecture} and string theory,
the \emph{quantum product} should be more natural and display
certain functoriality not available to the cup product
among $K$-equivalent manifolds.

\emph{Flops} are typical examples of $K$-equivalent birational
maps:
\begin{equation*}
\xymatrix{X\ar[rd]^\psi \ar@{-->}[rr]^f & &X'\ar[ld]_{\psi'} \\
&\bar X}
\end{equation*}
In fact they form the building blocks to connect birational minimal
models \cite{Ka}. The simplest flop is the simple $P^1$ flop
(Atiyah flop) in dimension 3. It is known that up to deformations
it generates, \emph{locally} or \emph{symplectically}, all
$K$-equivalent maps for threefolds. The quantum corrections by
extremal ray invariants to the cup product in the local 3 dimensional case was
first observed by Aspinwall--Morrison and Witten \cite{Witten} and
later globalized by Li--Ruan through the degeneration formula
\cite{LiRu}.

The higher dimensional generalizations are known as \emph{ordinary
$P^r$ flops} (also abbreviated as ``ordinary flops'' or ``$P^r$
flops''). The local geometry is encoded in a triple $(S, F, F')$ where $S$
is a smooth variety and $F$, $F'$ are two rank $r + 1$ vector
bundles over $S$. If $Z \subset X$ is the $f$-exceptional loci,
then $\bar\psi: Z \cong P(F) \to S \subset \bar X$ with fibers
spanned by the flopped curves $C \cong P^1$ and $N_{Z/X} =
\bar\psi^*F' \otimes \mathscr{O}_Z(-1)$. Similar structure holds
for $Z' \subset X'$, with $F$ and $F'$ exchanged.
See Section~\ref{s:1.1} for details.
(We note that the Atiyah flop corresponds to $S = \mbox{pt}$ and $r = 1$.)
Thus it is reasonable to expect that ordinary flops
play a vital role in the study of $K$-equivalent maps.
For example, up to complex cobordism, any $K$-equivalent map can be
decomposed into $P^1$ flops \cite{Wang3}.

The study of invariance of quantum product under ordinary flops in
higher dimensions was started in \cite{LLW}. The canonical
correspondence is given by the graph closure $[\bar \Gamma_f]$
and the quantum invariance under 
$$\T = [\bar \Gamma_f]_*: QH(X) \to QH(X')$$
is proved for all \emph{simple $P^r$ flops}, i.e.~with $S = \mbox{pt}$.
The crucial idea is to interpret \emph{$\T$-invariance} in terms of
\emph{analytic continuations in Gromov--Witten theory}.

Let us explain this point in a little more details.
We use \cite{CoKa} as our general reference for early developments
in Gromov--Witten invariants. Let $\overline{M}_{g, n}(X, \beta)$ be
the moduli space of stable maps from genus $g$ nodal curves with
$n$ marked points to $X$, and let $e_i: \overline{M}_{g, n}(X,
\beta) \to X$ be the evaluation maps. The Gromov--Witten potential
\begin{equation*}
\begin{split}
F^X_g(t) = \sum_{n, \beta} \frac{q^\beta}{n!} \langle
t^n \rangle^X_{g, n, \beta} = \sum_{n \ge 0,\, \beta \in NE(X)}
\frac{q^\beta}{n!} \int_{[\overline{M}_{g, n}(X, \beta)]^{vir}}
\prod_{i = 1}^n e_i^* t
\end{split}
\end{equation*}
is a formal function in $t \in H(X)$ and Novikov variables
$q^\beta$, with $\beta \in NE(X)$, the Mori cone of effective
classes of one cycles. Modulo convergence issues, it is a
function on the \emph{complexified K\"ahler cone} $\omega \in
\mathcal{K}^\mathbb{C}_X := H^{1, 1}_\mathbb{R} + i\mathcal{K}_X$
via
\begin{equation*}
q^\beta = e^{2\pi i (\beta.\omega)}.
\end{equation*}
Under the canonical correspondence $\T$,
$F_g^X$ and $F_g^{X'}$ share the same variable $t \in H \cong H(X,
\mathbb{C}) \cong H(X', \mathbb{C})$.
However, $\T$ does not identify $NE(X)$ with $NE(X')$.
Indeed, for the flopped curve classes $\ell = [C]$ (resp.~$\ell' = [C']$),
we have
\begin{equation*}
\T \ell = -\ell' \notin NE(X').
\end{equation*}
By duality this implies that
$\mathcal{K}^\mathbb{C}_X \cap \mathcal{K}^\mathbb{C}_{X'} =
\emptyset$ in $H^2_\mathbb{C}$.
Hence $F_g^X$ and $F_g^{X'}$ have different \emph{domains} and
comparison can only make sense after analytic continuations over
a certain compactification of ${\mathcal{K}^\mathbb{C}_X \cup
\mathcal{K}^\mathbb{C}_{X'}} \subset H^2_\mathbb{C}$.
(Thus the \emph{naive K\"ahler moduli} $\mathcal{K}$ is usually regarded
as the closure of the union of all $\mathcal{K}^\mathbb{C}_{X'}$'s
with $X' =_K X$.)
In other words, we set $\T q^\beta = q^{\T \beta}$.
Then $\T F_g^X$ can not be a formal GW potential of $X'$.

In this paper, we will focus on genus zero theory, which
carries a quantum product structure, or equivalently a Frobenius structure
\cite{yM}.
Let $\{T_{\mu}\}$ be a basis of $H$ and $\{T^{\mu} := \sum g^{\mu \nu} T_{\nu}\}$
the dual basis with respect to the Poincar\'e pairing,
where $g_{\mu \nu} = (T_{\mu}.T_{\nu})$ and $(g^{\mu \nu}) = (g_{\mu \nu})^{-1}$ is
the inverse matrix.
Denote $t = \sum t^{\mu} T_{\mu}$ a general element in $H$.
The \emph{big quantum ring} $(QH(X), *)$ uses only the genus
zero potential with 3 or more marked points:
\begin{equation*}
\begin{split}
T_{\mu} *_t T_{\nu} = \sum_{\kappa} \frac{\p^3 F^X_0}{\p t^{\mu} \p t^{\nu} \p t^{\kappa}}(t)
T^{\kappa} = \sum_{{\kappa},\,n \ge 0,\, \beta \in NE(X)} \frac{q^\beta}{n!} \langle T_{\mu},
T_{\nu}, T_{\kappa}, t^n \rangle^X_{0, n + 3, \beta} T^{\kappa}.
\end{split}
\end{equation*}
The \emph{Witten--Dijkgraff--Verlinde--Verlinde equation} (WDVV)
guarantees that $*_t$ is a family of associative products on $H$
parameterized by $t \in H$.
Equivalently, it equips $H$ a structure of
\emph{formal Frobenius manifold} $H_X$ with a family
(in $z  \in \mathbb{C}^\times$) of
integrable ($=$ flat) \emph{Dubrovin connections}
$$
\nabla^z = d -z^{-1} \sum_\mu dt^\mu \otimes T_\mu *_t
$$
on the tangent bundle $TH = H \times H$.

There is a natural embedding of $\mathcal{K}^\mathbb{C}_X$ in $H$.
With suitable choice of coordinates we have $q^\ell = e^{2\pi i
t_\ell}$ with the K\"ahler constraint ${\rm Im}\, t_\ell > 0$.
Since now $\T q^\ell = q^{-\ell'}$, $\{q^\ell, q^{\ell'}\}$ serve
as an atlas for $P^1$, the compactification of
$\mathbb{C}/\mathbb{Z} \cong \mathbb{C}^\times$. This gives the
formal $H$ an analytic $P^1$ direction. In \cite{LLW}, for simple
flops the structural constants $\p^3_{\mu\nu\kappa}F^X_0(t)$ for big quantum
product are shown to be analytic (in fact algebraic) in $q^\ell$.
Moreover, $\T$ identifies $H_X$ and $H_{X'}$ through analytic
continuations over this $P^1$. Based on this, in \cite{ILLW} the
Frobenius structure is further exploited to conclude analytic
continuations from $F^X_g$ to $F^{X'}_g$ for all simple flops and
for all $g \ge 0$.

\subsection{Outline of the contents}
This is the first of a sequence of papers proving the quantum invariance
under ordinary flops over a smooth base. 
In this first part, we determine the defect of the cup product 
under the canonical correspondence and show that it is corrected 
by the small quantum product attached to the extremal ray. 
We then perform various reductions to the local models.

In Part II \cite{LLWp2}, we show that the big quantum cohomology ring is 
invariant under analytic continuations in the K\"ahler moduli space 
for flops of splitting type.
In Part III \cite{LLQW}, the final part of this series, we remove the splitting condition by developing a quantum splitting principle, hence solve the problem completely.

In particular, this is the first result on the $K$-equivalence 
(crepant transformation) conjecture where the local structure of the 
exceptional loci can not be deformed to any explicit (e.g.~toric) geometry
and the analytic continuation is nontrivial.
As far as we know, this is also the first result for which the analytic
continuation is established with nontrivial Birkhoff factorizations.


We give an outline of the contents of this paper below.

\begin{conventions}
Throughout this paper, we work on the even cohomology
$H = H^{even}$ to avoid the complications on signs.
In particular, the degree always means the Chow degree.
Nevertheless all our discussions and results work for the full cohomology
spaces.
\end{conventions}

\subsubsection 
{Defect of cup product under the canonical correspondence}

Let $\{\tb_i\}$ be a basis of $H(S)$ with dual basis $\{\check \tb_i\}$.
Let $h = c_1(\mathscr{O}_Z(1))$ and $H_k = c_k(Q_{F})$ where $Q_F
\to Z = P(F)$ is the universal quotient bundle. Similarly we define
$h'$ and $H_k'$ on the $X'$ side. The $H_k$'s are of fundamental
importance since
\begin{equation*}
\T H_k = (-1)^{r-k} H'_k
\end{equation*}
and the dual basis of $\{\tb_i h^j\}$ in $H(Z)$ is given by $\{\check \tb_i H_{r-j}\}$.

\begin{theorem}[(Topological defect)] \label{top-defect}
Let $a_1, a_2, a_3 \in H(X)$ with $\sum \deg a_i = \dim X$. Then
\begin{align*}
&(\T a_1.\T a_2.\T a_3)^{X'} - (a_1.a_2.a_3)^X \\
&\qquad = (-1)^r \times \sum\nolimits_{i_*, j_*} (a_1.\check\tb_{i_1} H_{r-j_1})^X
(a_2.\check\tb_{i_2} H_{r-j_2})^X (a_3.\check\tb_{i_3} H_{r-j_3})^X \\
&\qquad\qquad\qquad\qquad\qquad \times (s_{j_1 + j_2 + j_3 - (2r + 1)}(F + F'^*)
\tb_{i_1} \tb_{i_2} \tb_{i_3})^S,
\end{align*}
where $s_i$ is the $i$-th Segre class.
\end{theorem}

\subsubsection 
{Quantum corrections attached to the flopping extremal rays}
We then proceed to calculate the \emph{quantum corrections} attached to the
flopping extremal ray $\mathbb{N}\ell$.
Using the calculation, we demonstrate that
the ``quantum corrected product'', combining the classical product
and the quantum deformation attached to the extremal ray,
is $\T$-invariant after the analytic continuation.

The stable map moduli for the extremal ray has a bundle
structure over $S$:
\begin{equation*}
\xymatrix{
\Mbar_{0, n}(P^r, d\ell)\ar[r] & \Mbar_{0, n}(Z,
d\ell) \ar[r]^>>>>>>{e_i} \ar[d]_{\Psi_n} & Z \ar[ld]^{\bar\psi}\\
& S &}
\end{equation*}
In this case, the GW invariants on $X$ are reduced to
\emph{twisted invariants} on $Z$ by certain obstruction bundles. We define
the fiber integral (see \S\ref{set-up} for the details of the notations)
\begin{equation*}
\Big\langle \prod\nolimits_{i = 1}^n h^{j_i} \Big\rangle_d^{/S} :=
\Psi_{n*} \Big(\prod\nolimits_{i = 1}^n e_i^* h^{j_i} . e(R^1 ft_*
e_{n + 1}^* N_{Z/X}) \Big) \in A^\m(S)
\end{equation*}
as a \emph{$\bar\psi$-relative invariant over $S$}, a cycle of
codimension $\m := \sum j_i - (2r + 1 + n - 3)$. The absolute
invariant is obtained by the pairing on $S$: For $\bart_i \in H(S)$,
\begin{equation*}
\langle\bart_1 h^{j_1}, \cdots, \bart_n h^{j_n} \rangle_d^X = \Big(\langle h^{j_1},
\cdots, h^{j_n} \rangle_d^{/S}.\prod\nolimits_{i = 1}^n \bart_i \Big)^S.
\end{equation*}

If $\m = 0$ then the invariant reduces to the simple case. This
happens for $n = 2$ since then $j_1 = j_2 = r$. Thus we may
calculate \emph{extremal functions} based on the 2-point case by (divisorial)
reconstruction. To state the result, let
\begin{equation*}
\f(q) := \frac{q}{1 - (-1)^{r + 1}q}
\end{equation*}
which satisfies the functional equation $\f(q) + \f(q^{-1}) =
(-1)^r$.

For 3-point functions, we show that $W_\m := \sum_{d \in
\mathbb{N}} \langle h^{j_1}, h^{j_2}, h^{j_3} \rangle_d^{/S} q^d$
with $1 \le j_i \le r$ lies in $A^\m(S)[\f]$ and is
independent of the choices of $j_i$'s.

\begin{theorem}[(Quantum corrections)] \label{q-correct}
The function $W_\m$ is the action on $\f$ by a Chern classes
valued polynomial in the operator $\delta = qd/dq$. 
(See Proposition~\ref{recursive}.)
It satisfies
\begin{equation*}
W_\m - (-1)^{\m + 1}W_\m' = (-1)^r s_\m(F + F'^*).
\end{equation*}
\end{theorem}

This implies that the topological defect is corrected by the
3-point extremal functions. The analytic continuation for $n \ge 4$
points follows by reconstruction.

\subsubsection{Degeneration analysis} 
The next step is to prove that the big quantum ring, involving all curve
classes, are $\T$-invariant.
As a first step, this statement is reduced to a corresponding one
on $f$-special descendent invariants on the \emph{projective local models}
\[
 X_{loc} :=  \tilde E = P(N_{Z/X} \oplus \mathscr{O}) \overset{p}{\to} Z
\]
{and}
\[
 X'_{loc} := \tilde E' = P(N_{Z'/X'} \oplus \mathscr{O}) \overset{p'}{\to} Z'
\]
by a \emph{degeneration analysis}.

To compare GW invariants of non-extremal classes, the
application of \emph{degeneration formula} and \emph{deformation to
the normal cone} are well suited for ordinary flops with base $S$.
It reduces the problem to local models with induced flop $f: \tilde E \dasharrow \tilde E'$. The
reduction has two steps. The first reduces the problem to
\emph{relative local} invariants $\langle A \mid \s, \mu
\rangle^{(\tilde{E}, E)}$ where $E \subset \tilde E$ is the
infinity divisor. The second is a further reduction back to
\emph{absolute} local invariants, with possibly descendent
insertions coupled to $E$, called \emph{$f$-special type}.

The local model $\bar p:= \bar\psi \circ p: \tilde E \to S$ and the
flop $f$ are all over $S$, with simple case as fibers. In
particular, the kernel of $\bar p_*: N_1(\tilde E) \to N_1(S)$ is
spanned by the $p$-fiber line class $\gamma$ and $\bar \psi$-fiber
line class $\ell$. $\T$ is compatible with $\bar p$. Namely
\begin{equation*}
\xymatrix{ N_1(\tilde E) \ar[rr]^\T \ar[rd]_{\bar p_* \oplus d_2}
& &
N_1(\tilde E') \ar[ld]^{\bar p'_* \oplus d_2'} \\
& N_1(S) \oplus \mathbb{Z} &}
\end{equation*}
is commutative. Here we write a class $\beta$ in $N_1(\tilde E)$ as $\beta_S+d\ell+d_2\gamma$ with some $\beta_S$ in $N_1(S)$ and $d,d_2\in\mathbb{Z}$. Thus the functional equation of a generating series
$\langle A \rangle$ is equivalent to those of its various subseries
(fiber series) $\langle A \rangle_{\beta_S, d_2}$ labeled by
$NE(S) \oplus \mathbb{Z}$.

\begin{theorem}[(Degeneration reduction)] \label{deg-red}
To prove $\T \langle \alpha \rangle^X_g \cong \langle \T\alpha
\rangle^{X'}_g$ for all $\alpha \in H(X)^{\oplus n}$, $g \le g_0$, it is enough
to prove the local case $f: \tilde E \to \tilde E'$ for descendent
invariants of $f$-special type:
\begin{equation*}
\T\langle A, \tau_{k_1} \s_1, \ldots, \tau_{k_\rho} \s_\rho
\rangle^{\tilde{E}}_{g, \beta_S, d_2} \cong \langle \T A, \tau_{k_1}
\s_1, \ldots, \tau_{k_\rho} \s_\rho \rangle^{\tilde{E}'}_{g, \beta_S,
d_2}
\end{equation*}
for any $A \in H(\tilde E)^{\oplus n}$, $k_j \in \mathbb{N}\cup
\{0\}$, $\s_j \in H(E) \subset H(\tilde{E})$,
$g \le g_0$, $\beta_S \in NE(S)$ and $d_2 \ge 0$.
\end{theorem}

\subsubsection 
{Further reduction to the big quantum ring/quasi-linearity on the local models}
While the degeneration reduction works for higher genera,
for $g = 0$ more can be said.
Using the \emph{topological recursion relation} (TRR) and the
\emph{divisor axiom} (for descendent invariants),
the $\T$-invariance for $f$-special invariants can be completely reduced to
the $\T$-invariance of big quantum rings for local models (Theorem \ref{comp-red-local}).

We then employ the \emph{divisorial reconstruction} \cite{LP}
and the WDVV equation to make a further reduction to
an $\T$-invariance statement about \emph{elementary} $f$-special invariants with at most one special insertion.

To state the result, we assume now $X = X_{loc}
= \tilde E$. Since $X \to S$ is a double projective bundle, $H(X)$
is generated by $H(S)$ and the relative hyperplane classes $h$ for
$Z \to S$ and $\xi$ for $X \to Z$. This leads to another useful reduction: By moving all the classes $h$, $\xi$ and $\psi$ into the last insertion (divisorial reconstruction), the problem is reduced to the case
\begin{equation*}
\langle \bart_1,\ldots,\bart_{n-1}, \bart_n \tau_k h^j\xi^i \rangle_{\beta_S, d_2}^X
\end{equation*}
with $\bart_l \in H(S)$, $d_2 \in \mathbb{Z}$, where $k \ne 0$ only if
$i \ne 0$.

By a further application of WDVV equations, the $\T$-invariance can always be reduced to the case $i \ne 0$ even if $k = 0$. Since $\xi$ is the class of infinity divisor which is within the isomorphism loci of the flop,
such an $\T$-invariance statement is intuitively plausible.
We call it the \emph{type I quasi-linearity} property 
(c.f.~Theorem \ref{QL-I}). \smallskip

The above steps furnish a complete reduction to projective local models $X_{loc}$, which works for any $F$ and $F'$.
\medskip

To proceed, notice that these descendent invariants are encoded by their
generating function, i.e. the so called (big) $J$ function:
For $\tau \in H(X)$,
\begin{equation*}
  J^X(\tau, z^{-1}) := 1 + \frac{\tau}{z} + \sum_{\beta, n, \mu} \frac{q^{\beta}}{n!}
  T_{\mu} \left\langle \frac{T^{\mu}}{z(z-\psi)}, \tau, \cdots, \tau \right\rangle_{0,n+1, \beta}^X.
\end{equation*}
The determination of $J$ usually relies on the existence of $\Bbb
C^\times$ actions. Certain localization data $I_\beta$ coming from
the stable map moduli are of hypergeometric type. For ``good''
cases, say $c_1(X)$ is semipositive and $H(X)$ is generated by
$H^2$, $I(t) = \sum I_\beta\, q^\beta$ determines $J(\tau)$ on the
small parameter space $H^0 \oplus H^2$ through the ``classical''
\emph{mirror transform} $\tau = \tau(t)$. For a simple flop, $X =
X_{loc}$ is indeed semi-Fano toric and the classical Mirror Theorem
(of Lian--Liu--Yau and Givental) is sufficient \cite{LLW}. (It
turns out that $\tau = t$ and $I = J$ on $H^0 \oplus H^2$.)

For general base $S$ with given $QH(S)$, the determination of $QH(P)$ for
a projective bundle $P \to S$ is far more involved. To allow \emph{fiberwise localization} to determine the structure of GW
invariants of $X_{loc}$, the bundles $F$ and $F'$ are then assumed to be split bundles. This is main subject to be studied in Part II of this series \cite{LLWp2}.

\begin{remark}
Results in this paper had been announced, in increasing degree of generalities,
by the authors in various conferences during 2008-2010; 
see e.g.~\cite{Lin, Wang4, LLW2} where more example-studies can be found.
\end{remark}

\subsection{Acknowledgements}
Y.-P.~Lee is partially supported by the NSF;
H.-W.~Lin is partially supported by the MOST;
C.-L.~Wang is partially supported by the MOST and the MOE.
We are particularly grateful to Taida Institute of Mathematical Sciences (TIMS)
for its steady support which makes this long-term collaborative project possible.
Finally, we would like to thank the anonymous referee for pointing out several typographical errors in an earlier version of the paper.

\section{Defect of the classical product} \label{defect}

\subsection{Cohomology correspondence for $P^r$ flops.} \label{s:1.1}
We recall the construction of ordinary flops in \cite{LLW} to fix
notations.

Let $X$ be a smooth complex projective manifold and $\psi:X\to \bar
X$ a flopping contraction in the sense of minimal model theory,
with $\bar\psi: Z\to S$ the restriction map on the exceptional
loci. Assume that
\begin{itemize}
\item[(i)] $\bar\psi$ equips $Z$ with a $P^r$-bundle structure
    $\bar\psi : Z = P(F) \rightarrow S$ for some rank $r + 1$
    vector bundle $F$ over a smooth base $S$, \item[(ii)]
    $N_{Z/X}|_{Z_s} \cong \mathscr{O}_{P^r}(-1)^{\oplus(r +
    1)}$ for each $\bar\psi$-fiber $Z_s$, $s\in S$.
\end{itemize}
Then there is another rank $r + 1$ vector bundle $F'$ over $S$ such
that
\begin{equation*}
N_{Z/X}\cong \mathscr{O}_{P(F)}(-1)\otimes
\bar\psi^*F'.
\end{equation*}

We may blow up $X$ along $Z$ to get $\phi : Y \rightarrow X$. The
exceptional divisor
\begin{equation*}
E = P(N_{Z/X})\cong P(\bar\psi^*F') =
\bar\psi^*P(F')=P(F)\times_SP(F')
\end{equation*}
is a $P^r \times P^{r}$-bundle over $S$. We may then blow down $E$
along another fiber direction $\phi' : Y \rightarrow X'$ to get
another contraction $\psi': X' \to \bar X$, with exceptional loci
$\bar\psi' : Z' = P(F') \rightarrow S$ and
$N_{Z'/X'}|_{\bar\psi'-{\rm fiber}} \cong
\mathscr{O}_{P^{r}}(-1)^{\oplus(r + 1)}$.

We call the $f : X \dashrightarrow X'$  an \emph{ordinary $P^r$
flop}. The various sets and maps are summarized in the following
commutative diagram. \small
\begin{equation*} 
\xymatrix{ &
E\;\ar[ld]_>>>>{\bar\phi}\ar[rd]|\hole^>>>>{\bar\phi'}
\ar@{^{(}->}[r]^j &
Y\;\ar[ld]_>>>>\phi\ar[rd]^>>>>{\phi'}\\
Z\;\ar[rd]_>>>>>>>>{\bar\psi}\ar@{^{(}->}[r]^i &
X\;\ar[rd]_>>>>>>>>>>{\psi} &
Z'\;\ar[ld]|\hole^>>>>>>>>>>{\bar\psi'} \ar@{^{(}->}[r]^{i'} &
X'\;\ar[ld]^>>>>>>>>{\psi'}\\
&S\;\ar@{^{(}->}[r]^{j'}&\overline{X}& }
\end{equation*} \normalsize
where the normal bundle of $E$ in $Y$ is
\begin{equation*}
N_{E/Y} = \bar\phi^*\mathscr{O}_{P(F)}(-1)\otimes
\bar\phi'^*\mathscr{O}_{P(F')}(-1).
\end{equation*}

\medskip

First of all, we have found a \emph{canonical} correspondence
between the cohomology groups of $X$ and $X'$.

\begin{theorem} [\cite{LLW}] \label{t:1.1}
For an ordinary $P^r$ flop $f:X\dashrightarrow X'$, the graph
closure $T := [\bar\Gamma_f]\in A(X \times X')$ identifies the
Chow motives $\hat{X}$ of $X$ and $\hat{X'}$ of $X'$, i.e.\ $\hat X
\cong \hat X'$ via $T^t\circ T = \Delta_X$ and $T\circ T^t =
\Delta_{X'}$. In particular, $\T := T_*: H(X) \to H(X')$ preserves
the Poincar\'e pairing on cohomology groups.
\end{theorem}

In practice, the correspondence $T$ associates a map on Chow
groups:
\begin{equation*}
\T: A(X)\to A(X');\quad W\mapsto p'_*(\bar\Gamma_f.p^*W) =
\phi'_*\phi^*W
\end{equation*}
where $p$ (resp. $p'$) is the projection map from $X \times X'$ to
$X$ (resp. $X'$).

Secondly, parallel to the procedure in \cite{LLW}, we need to
determine the explicit formulae for the associated map $\T$
restricted to $A(Z)$. The Leray--Hirsch theorem says that
\begin{equation*}
A(Z) = \bar\psi^* A(S)[h]/f_F(h)
\end{equation*}
where $f_F(\lambda) = \lambda^{r + 1} + \bar \psi^* c_1(F)\lambda^r
+ \cdots +\bar\psi^* c_{r + 1}(F)$ is the Chern polynomial of $F$
and $h = c_1(\mathscr{O}_{P(F)}(1))$. Thus a class $\alpha \in
A(Z)$ has the form $\alpha = \sum_{i = 0}^r h^i\bar\psi^*a_i$ for
some $a_i\in A(S)$.

By the pull-back formula from the intersection theory, it is
easy to see that for $a \in A_k(Z)$ we have
\begin{equation*}
\phi^*(i_* a) = j_*\Big(c_r(\mathscr{E}).\bar\phi^* a \Big) \in A_k(Y)
\end{equation*}
where $\mathscr{E}$ is the excess normal bundle defined by
\begin{equation*}
 0\to N_{E/Y}\to \phi^* N_{Z/X} \to \mathscr{E}\to 0.
\end{equation*}

By the functoriality of pull-back and push-forward together with
the above formula, we can conclude from $\T (i_*(\sum
h^i\bar\psi^*a_i)) =  \sum i'_*(\T(i_*(h^i)).i'_*\bar\psi'^*a_i)^{Z'}$ that $\T$
restricted to $A(Z)$ is $A(S)$-linear. Here we identify the
ring $A(S)$ with its isomorphic images in $A(Z)$ and $A(Z')$
via $\bar\psi^*$ and $\bar\psi'^*$ respectively.

Under such an identification, we will abuse notations to denote
$c_i(F)$, $\bar\psi^* c_i(F)$ and $\bar\psi'^* c_i(F)$ by the same
symbol $c_i$. Similarly we denote $c_i(F')$, $\bar\psi^* c_i(F')$
and $\bar\psi'^* c_i(F')$ by $c_i'$. We use this abbreviation for
any class in $A(S)$. And for $\alpha \in A(Z)$ we often omit
$i_*$ from $i_*\alpha$ when $\alpha$ is regarded as a class in
$A(X)$, unless possible confusion should arise. Similarly, we do
these for $\alpha' \in A(Z') \hookrightarrow A(X')$.

The $A(S)$-linearity of $\T$ restricted to $A(Z)$ allows us to
focus on the study of a basis for $A(Z)$ over $A(S)$. Recall
that for a simple $P^r$ flop we have the basic transformation
formula $\T(h^k) = (-1)^{r-k}  h'^k$. Unfortunately, for a general
$P^r$ flop, this does not hold any more, so a better candidate has
to be sought out.

Note that the key ingredient in the pull-back formula is
$c_r(\mathscr{E})$. From the Euler sequence
\begin{equation*}
0\to \mathscr{O}_{Z'}(-1)\to \bar\psi'^* F' \to Q_{F'}
\to 0
\end{equation*}
and the short exact sequence defining the excess normal bundle
$\mathscr{E}$, we get $\mathscr{E} =
\bar\phi^*\mathscr{O}_{P(F)}(-1) \otimes \bar\phi'^*Q_{F'}$. A
simple computation leads to
\begin{equation*}
c_r(\mathscr{E}) = (-1)^r(\bar\phi^*h^r -
\bar\phi'^*H'_1\bar\phi^*h^{r - 1} +
\bar\phi'^*H'_2\bar\phi^*h^{r-2} + \cdots +
(-1)^r\bar\phi'^*H'_r),
\end{equation*}
where $H'_k = c_k(Q_{F'})$. Explicitly,
\begin{equation*}
H'_k = h'^k + c_1'h'^{k - 1} + \cdots + c_k'
\end{equation*}
where $h' = c_1(\mathscr{O}_{P(F')}(1))$. Similarly, we denote
\begin{equation*}
H_k = c_k(Q_{F}) = h^k + c_1h^{k - 1} + \cdots + c_k.
\end{equation*}
Notice that $H_k = 0 = H'_k$ for $k > r$. Finally, we find that
$H_k$, $H'_k$ turn out to be the correct choice.

\begin{proposition}\label{H}
For all positive integers $k \leq r$,
\begin{equation*}
\T(H_k) = (-1)^{r-k}  H'_k.
\end{equation*}
\end{proposition}

\begin{proof}
First of all, we have the basic identities: $h^{r+1} + c_1 h^r +
\cdots + c_{r + 1} = 0$, $\bar\phi'_* \bar\phi^*h^i = 0$ for all $i
< r$ and $\bar\phi'_* \bar\phi^* h^r = [Z']$. The latter two follow
from the definitions and dimension consideration.

In order to determine $\T( H_k) =
\bar\phi'_*(c_r(\mathcal{E}).\bar\phi^*H_k)$, we need to take care
of the class
$$
\bar\phi'_*(\bar\phi'^*H'_{r-i}\bar\phi^*h^i.\bar\phi^*H_k)
$$ 
with $0 \leq i \leq r$, here $H'_0 := 1$.

If $i > r - k$, then
\begin{align*}
&\bar\phi'_*(\bar\phi'^*H'_{r-i}\bar\phi^*h^i.\bar\phi^*H_k)
= \bar\phi'_*(\bar\phi'^*H'_{r-i}\bar\phi^*(h^{k+i} + c_1h^{k +i- 1}
+ \cdots + c_kh^i))\\
&= -\bar\phi'_*(\bar\phi'^*H'_{r-i}\bar\phi^*(c_{k+1}h^{i-1} +
c_{k+2}h^{i-2} +\cdots + c_{r+1}h^{i+k-r-1})) = 0
\end{align*}
since the power in $h$ is at most $i - 1 < r$.

If $i < r - k$, then again $\bar\phi'_*(\bar\phi'^*H'_{r-i}
\bar\phi^*h^i.\bar\phi^*H_k) = 0$ since the power in $h$ is at most
$i + k < r$.

For the remaining case $i = r - k$,
\begin{equation*}
\bar\phi'_*(\bar\phi'^*H'_{r-i}\bar\phi^*h^i.\bar\phi^*H_k)
= \bar\phi'_*(\bar\phi'^*H'_{r-i}\bar\phi^*h^r) = H'_{r-i} = H'_k.
\end{equation*}
We conclude that
\begin{equation*}
\T( H_k) = (-1)^r  \sum_{i=0}^r(-1)^{r-i}
\bar\phi'_*(\bar\phi'^*H'_{r-i}\bar\phi^*h^i.\bar\phi^*H_k) =
(-1)^{r-k} H'_k.
\end{equation*}
\end{proof}

\begin{remark}
Unlike simple $P^r$ flops, here the image class of $h^k$ under $\T$
looks more complicated. As a simple corollary of the above
proposition, we may show, by induction on $k$, that for all $k\in
\mathbb{N}$,
\begin{equation*}
\T( h^k) = (-1)^{r - k}  (a_0h'^k + a_1h'^{k - 1} + \cdots
+ a_k) \in A(Z')
\end{equation*}
where $a_0 = 1$ and $a_k \in A(S)$ are determined by the
recursive relations:
\begin{equation*}
c_k' = a_k - c_1 a_{k - 1} + c_2 a_{k - 2} + \cdots + (-1)^k c_k.
\end{equation*}
And symmetrically
\begin{equation*}
\T^*( h'^k) = (-1)^{r - k}  (a'_0h^k + a'_1h^{k - 1} +
\cdots + a'_k) \in A(Z)
\end{equation*}
with $a'_0=1,\ \ a'_k = c'_1 a'_{k - 1} - c'_2 a'_{k -2} + \cdots +
(-1)^{k - 1}c'_k + c_k $.

To put these formulae into perspective, we consider the virtual
bundles
\begin{equation*}
A := F' - F^*;\qquad A' := F - F'^*.
\end{equation*}
Then $a_k = c_k(A)$ and $a_k' = c_k(A')$. Notice that since $a_k$
and $a_k'$ are Chern classes of virtual bundles, they may survive
even for $k \ge r + 1$.

It is also interesting to notice that the explicit formula reduces
to
\begin{equation*}
\T ( h^k) = (-1)^{r - k}  h'^k
\end{equation*}
without lower order terms precisely when $F' = F^*$, the dual of
$F$.
\end{remark}

\subsection{Triple product} \label{3-prod}
Let $\{\tb_i^k\}$ be a basis of
$H^{2k}(S)$ and $\{\hatt \tb_i^k \}\subset H^{2(s-k)}(S)$
be its dual basis where $s = \dim S$.
It is an easy but quite crucial discovery that
the dual basis of the canonical basis $\{\tb_i^kh^j\}$ in $H(Z)$
can be expressed in terms of $\{H_k\}_{k\geq 0}$.

\begin{lemma}\label{dual-basis}
The dual basis of $\{\tb_i^{k-j}h^j\}_{j\leq \min\{k,r\}}$ in
$H^{2k}(Z)$ is $\{\hatt \tb_i^{k-j} H_{r-j}\}_{j\leq \min\{k,r\}}$ in
$H^{2(r + s - k)}(Z)$.
\end{lemma}

\begin{proof}
We have to check that $(\tb_i^{k-j}h^j.\hatt \tb_i^{k-j} H_{r-j}) = 1$
and $(\tb_i^{k-j}h^j.\hatt \tb_i^{k-j'} H_{r-j'}) = 0$ for any $j\ne j'$.
Indeed,
\begin{equation*}
(\tb_i^{k-j}h^j.\hatt \tb_i^{k-j} H_{r-j}) = \tb^s(h^r + c_1h^{r-1} +
\cdots)= \tb^sh^r = 1
\end{equation*}
since $\tb^sc_i = 0$ for all $i\geq 1$ by degree consideration.

\begin{notations}
When $X$ is a bundle over $S$, classes in $H(S)$ may be considered as
classes in $H(X)$ by the obvious pullback, which we often omit
in the notations.
To avoid confusion, we consistently employ the notation $\check \tb_i$
as the dual class of $\tb_i \in H(S)$
with respect to the Poincar\'e pairing in $S$.
The ``raised'' index form, e.g.~$T^{\mu}$ as the dual of $T_{\mu} \in H(X)$,
is reserved for duality with respect to the Poincar\'e pairing in $X$.
\end{notations}

Now if $j'> j$ then
$$k-j + (s-(k-j')) = s + (j'-j) > s,$$
which implies that $\tb_i^{k-j}\hatt \tb_i^{k-j'} =0$. Conversely, if
$j' < j$ then $\tb_i^{k-j}\hatt \tb_i^{k-j'} \in H^{2(s-(j-j'))}(S)$ and
\begin{align*}
h^jH_{r-j'} &= h^{r+(j-j')} + c_1h^{r+(j-j')-1}
+ \cdots + c_{r-j'}h^j \\
&= - c_{r-j'+1}h^{j-1} - \cdots - c_{r+1}h^{j-j'-1}.
\end{align*}
Again since
$$(s-(j-j'))+(r-j'+z) = s+(r+z-j) > s$$
for $z\geq 1$, we
have $\tb_i^{k-j}\hatt \tb_i^{k-j'} c_{r-j'+z}h^{j-z-1} = 0$. The result
follows.
\end{proof}

Now we can determine the difference of the pullback classes of
$a$ and $\T a$ as follows.

\begin{proposition}
For a class $a \in H^{2k}(X)$, let $a' = \T a$ in $X'$.
Then
\begin{equation*}
\phi'^* a' = \phi^* a + j_*\sum_i \sum_{1\leq j \leq
\min\{k,r\}}(a.\hatt \tb_i^{k-j} H_{r-j})\tb_i^{k-j}\frac{x^j-(-y)^j}{x + y}
\end{equation*}
where $x = \bar\phi^*h$, $y = \bar\phi'^*h'$.
\end{proposition}

\begin{proof}
Recall that
$$N_{E/Y} = \bar\phi^*\mathscr{O}_{Z}(-1)\otimes \bar\phi'^*\mathscr{O}_{Z'}(-1)$$
and hence $c_1(N_{E/Y}) = -(x + y)$.
Since the difference $\phi'^* a' - \phi^* a$ has
support in $E$, we may write $\phi'^* a' - \phi^* a =
j_*\lambda$ for some $\lambda \in H^{2(k - 1)}(E)$. Then
\begin{equation*}
(\phi'^* a' - \phi^* a)|_E = j^*j_*\lambda =
c_1(N_{E/Y})\lambda = -(x+y)\lambda.
\end{equation*}

Notice that while the inclusion-restriction map $j^*j_*$ on
$H(E)$ may have non-trivial kernel, elements in the kernel never
occur in $\phi'^* a' - \phi^* a$ by the Chow moving lemma.
Indeed if $j^*j_* \lambda \equiv j_*\lambda|_E = 0$ then
$j_*\lambda$ is rationally equivalent to a cycle $\lambda'$
disjoint from $E$. Applying $\phi'_*$ to the equation
$$\phi'^* a' - \phi^* a = j_*\lambda \sim \lambda'$$
gives rise to
\begin{equation*}
\phi'_*\lambda' \sim \phi'_*\phi'^* a' - \phi'_*\phi^* a =
 a' - a' = 0.
\end{equation*}
This leads to $\lambda' \sim 0$ on $Y$.

Hence
\begin{equation*}
\lambda = -\frac{1}{x+y}((\phi'^* a')|_E - (\phi^* a)|_E) =
-\frac{1}{x+y}(\bar\phi'^*( a'|_{Z'}) - \bar\phi^*( a|_Z)).
\end{equation*}
By the above lemma, we get
\begin{align*}
\bar\phi^*( a|_Z) &= \bar\phi^*(\sum_i \sum_{j \leq
\min\{k,r\}}( a.\hatt \tb_i^{k-j} H_{r-j})\tb_i^{k-j}h^j)\\
&= \sum_i\sum_{j \leq \min\{k,r\}}( a.\hatt \tb_i^{k-j} H_{r-j})\tb_i^{k-j}x^j.
\end{align*}
Similarly, we have
\begin{equation*}
\bar\phi'^*( a'|_{Z'}) = \sum_i\sum_{j \leq \min\{k,r\}}
( a'.\hatt \tb_i^{k-j} H'_{r-j})\tb_i^{k-j}y^j.
\end{equation*}
Since $\T$ preserves the Poincar\'e pairing,
\begin{equation*}
( a'.\hatt \tb_i^{k-j} H'_{r-j}) =
(\T a. \T((-1)^{r-(r-j)}\hatt \tb_i^{k-j} H_{r-j})) =
(-1)^j( a. \hatt \tb_i^{k-j} H_{r-j}).
\end{equation*}
Putting these together, we obtain
\begin{equation*}
\lambda = \sum_i\sum_{1\leq j \leq \min\{k,r\}}( a.\hatt
\tb_i^{k-j} H_{r-j})\tb_i^{k-j}\frac{x^j-(-y)^j}{x + y}.
\end{equation*}
\end{proof}

\begin{remark}
Notice that since the power in $x$ (and in $y$) is at most $r - 1$,
the class $\lambda$ clearly contains non-trivial $\bar\phi$ and
$\bar\phi'$ fiber directions. Thus this proposition in particular
gives rise to an alternative proof of equivalence of Chow motives
under ordinary flops (Theorem 0.2). Indeed this is precisely the
quantitative version of the original proof in \cite{LLW}.
\end{remark}

Now we may compare the triple products of classes in $X$ and $X'$.

\begin{theorem}[($=$ Theorem \ref{top-defect})] \label{top-defect-1}
Let $ a_i \in H^{2k_i}(X)$ for $i=1, 2, 3$ with $k_1 + k_2 + k_3
= \dim X = s + 2r + 1$. Then
\begin{align*}
(\T a_1.\T a_2.\T a_3) &= ( a_1.a_2.a_3)
+ (-1)^r \times\\
&\qquad\sum ( a_1.\hatt \tb_{i_1}^{k_1-j_1}
H_{r-j_1})( a_2.\hatt \tb_{i_2}^{k_2-j_2}  H_{r-j_2})
( a_3.\hatt \tb_{i_3}^{k_3-j_3}
H_{r-j_3}) \times \\
&\qquad\qquad\qquad (\tilde s_{j_1 + j_2 + j_3 - 2r - 1}
\tb_{i_1}^{k_1-j_1}\tb_{i_2}^{k_2-j_2} \tb_{i_3}^{k_3-j_3}),
\end{align*}
where the sum is over all possible $i_1, i_2, i_3$ and $j_1, j_2,
j_3$ subject to constraint: $1\leq j_p \leq \min\{r,k_p\}$ for
$p=1, 2, 3$ and $j_1+j_2+j_3\geq 2r + 1$. Here
\begin{equation*}
\tilde s_i := s_i(F + F'^*)
\end{equation*}
is the $i$th Segre class of $F + F'^*$.
\end{theorem}

\begin{proof}
First of all, $\phi'^*\T a_i = \phi^* a_i + j_*\lambda_i$
for some $\lambda_i \in H^{2(k_i - 1)}(E)$ which contains both fiber
directions of $\bar\phi$ and $\bar\phi'$. Hence
\begin{align*}
&(\T a_1.\T a_2.\T a_3) =
(\phi'^*\T a_1.\phi'^*\T a_2.(\phi^* a_3 +
j_*\lambda_3)) \\
&\quad =
(\phi'^*\T a_1.\phi'^*\T a_2.\phi^* a_3) = ((\phi^* a_1 + j_*\lambda_1) . (\phi^* a_2 +
j_*\lambda_2) .\phi^* a_3).
\end{align*}

Among the resulting terms, the first term is clearly equal to
$( a_1.a_2.a_3)$.

For those terms with two pull-backs like
$\phi^* a_1.\phi^* a_3$, the intersection values are zero
since the remaining part necessarily contains nontrivial $\bar\phi$
fiber direction.

The terms with $\phi^* a_3$ and two exceptional parts
contribute
\begin{align*}
&\phi^* a_3.j_*\tb_{i_1}^{k_1-j_1}\left(\frac{x^{j_1} -
(-y)^{j_1}}{x + y}\right).j_*\tb_{i_2}^{k_2-j_2}
\left(\frac{x^{j_2} - (-y)^{j_2}}{x + y}\right)\\
&= -\phi^* a_3.j_*
\big(\tb_{i_1}^{k_1-j_1}\tb_{i_2}^{k_2-j_2}(x^{j_1}
-(-y)^{j_1})(x^{j_2 - 1} +x^{j_2 - 2}(-y) + \cdots + (-y)^{j_2 -
1})\big)
\end{align*}
times $( a_1.\hatt \tb_{i_1}^{k_1-j_1} H_{r-j_1}) ( a_2.\hatt
\tb_{i_2}^{k_2-j_2} H_{r-j_2})$. The terms with non-trivial
contribution must contain $y^q$ with $q\geq r$ which implies
$j_1+j_2 - 1\geq r$, hence such terms are
\begin{equation*}
-(-y)^{j_1}(x^{j_2 - 1 -(r - j_1)}(-y)^{r - j_1} + x^{j_2 - 1 -(r
- j_1)-1}(-y)^{r - j_1+1} + \cdots + (-y)^{j_2-1})
\end{equation*}
and the contribution after taking $\phi_*$ is
\begin{equation*}
(-1)^{r+1}(h^{j_1+j_2-r-1} - h^{j_1+j_2-r-2}s'_1 + \cdots +
(-1)^{j_1+j_2-r-1}s'_{j_1+j_2-r-1})
\end{equation*}
where $s'_i := s_i(F')$ is the $i$th Segre class of $F'$. Here we
use the property of Segre classes to obtain $\phi_*y^q = s'_{q-r}$
for $q\geq r+1$.

In terms of bundle-theoretic formulation,
\begin{align*}
&h^{j_1+j_2-r-1} - h^{j_1+j_2-r-2}s'_1 + \cdots +
(-1)^{j_1+j_2-r-1} s'_{j_1+j_2-r-1}\\
&= \big((1-s_1'+s_2'+\cdots)(1+h+h^2+\cdots)\big)_{j_1+j_2-r-1}\\
&= \left(s(F'^*)\frac{1}{(1-h)}\right)_{j_1+j_2-r-1} =
\left(\frac{c(F)}{(1-h)}s(F)s(F'^*) \right)_{j_1+j_2-r-1}\\
&= \big(c(Q_F).s(F+F'^*)\big)_{j_1+j_2-r-1}\\
&= H_{j_1+j_2-r-1} + H_{j_1+j_2-r-2}\tilde s_1 + \cdots + \tilde
s_{j_1+j_2-r-1}.
\end{align*}

With respect to the basis $\{\hatt \tb^k_i \}$, $\tilde s_p
\tb_{i_1}^{k_1-j_1} \tb_{i_2}^{k_2-j_2}$ is of the form
\begin{equation*}
\sum_{i_3} \big(\tilde s_p\tb_{i_1}^{k_1-j_1}\tb_{i_2}^{k_2-j_2}
\tb_{i_3}^{k_3-(2r+1+p-j_1-j_2)}\big) \hatt
\tb_{i_3}^{k_3-(2r+1+p-j_1-j_2)} .
\end{equation*}

We define the new index $j_3 = 2r+1+p-j_1-j_2$ and thus
$j_1+j_2+j_3\geq 2r+1$, also $p= j_1+j_2+j_3-2r-1$.\par By summing
all together, we get the result.
\end{proof}

There is a particularly simple case where no $H_i$ or Segre classes
$\tilde s_i$ are needed in the defect formula, namely the $P^1$
flops.

\begin{corollary}
For $P^1$ flops over any smooth base $S$ of dimension $s$, let
$ a_i \in H^{2k_i}(X)$ for $i=1, 2, 3$ with $k_1 + k_2 + k_3 =
\dim X = s + 3$. Then
\begin{equation*}
(\T a_1.\T a_2.\T a_3) = ( a_1.a_2.a_3)
- \sum ( a_1.\hatt \tb_1 )( a_2.\hatt \tb_2 ) ( a_3.\hatt \tb_3 )
(\tb_1 \tb_2 \tb_3)
\end{equation*}
with $\tb_i$ running over all basis classes in $H^{2(k_i - 1)}(S)$.
\end{corollary}

There is a trivial but useful observation on when the product is preserved:

\begin{corollary}
For a $P^r$ flop $f: X \dasharrow X'$, $a_1 \in H^{2k_1}(X)$, $a_2 \in H^{2k_2}(X)$ with $k_1 + k_2 \le r$, then $\T (a_1.a_2) = \T a_1.\T a_2$.
\end{corollary}

This follows from Theorem \ref{top-defect-1} since all the correction terms vanish for any $a_3$.
In fact it is a consequence of dimension count.

\section{Quantum corrections attached to the extremal ray} \label{q-correction}

\subsection{The set-up with nontrivial base}\label{set-up}

Let $a_i \in H^{2k_i}(X)$, $i = 1, \ldots, n$, with
\begin{equation*}
\sum_{i = 1}^n k_i = 2r + 1 + s + (n - 3).
\end{equation*}

Since
\begin{equation*}
a_i|_Z = \sum_{s_i} \sum_{j_i \leq \min\{k_i,
r\}}(a_i.\hatt \tb_{s_i}^{k_i - j_i} H_{r - j_i}) \tb_{s_i}^{k_i -
j_i}h^{j_i},
\end{equation*}
we compute
\begin{align*}
&\langle a_1, \ldots, a_n \rangle_{0, n, d\ell}^X \\
&= \sum_{\vec s, \vec j} \int_{\Mbar_{0, n}(Z, d\ell)} \prod_{i = 1}^n
\Big((a_i.\hatt \tb_{s_i}^{k_i - j_i} H_{r - j_i})\; e_i^*
(\bar\psi^* \tb_{s_i}^{k_i - j_i}.
h^{j_i})\Big).e(R^1 ft_* e_{n + 1}^* N)\\
&= \sum_{\vec s, \vec j} \prod_{i = 1}^n (a_i.\hatt
\tb_{s_i}^{k_i - j_i} H_{r - j_i}) \left[\prod_{i = 1}^n \tb_{s_i}^{k_i
- j_i}. \Psi_{n*} \Big(\prod_{i = 1}^n e_i^* h^{j_i}. e(R^1 ft_*
e_{n + 1}^* N) \Big)\right]^S,
\end{align*}
with the sum over all $\vec s= (s_1. \ldots, s_n)$ and admissible
$\vec j = (j_1, \ldots, j_n)$. By the fundamental class axiom, we
must have $j_i \ge 1$ for all $i$.

Here we make use of
\begin{equation*}
[\Mbar_{0, n}(X, d\ell)]^{virt} = [\Mbar_{0, n}(Z, d\ell)] \cap e(R^1 ft_*
e_{n + 1}^* N)
\end{equation*}
and the fiber bundle diagram over $S$
\begin{equation*}
\xymatrix{& \Mbar_{0, n + 1}(Z, d\ell) \ar[rd]^{e_{n + 1}} \ar[d]_{ft}
& N = N_{Z/X} \ar[d]\\
\Mbar_{0, n}(P^r, d\ell)\ar[r] &  \Mbar_{0, n}(Z,
d\ell) \ar[r]^{e_i} \ar[d]_{\Psi_n} & Z \ar[ld]^{\bar\psi}\\
& S &}
\end{equation*}
as well as the fact that classes in $S$ are constants among bundle
morphisms (by the projection formula applying to $\Psi_n = \bar\psi
\circ e_i$ for each $i$).

We must have $\sum (k_i - j_i) \le s$ to get nontrivial invariants.
That is,
\begin{equation*}
\sum_{i = 1}^n j_i \ge 2r + 1 + n - 3.
\end{equation*}

If the equality holds, then $\prod_{i = 1}^n \tb_{s_i}^{k_i - j_i}$
is a zero dimensional cycle in $S$ and the invariant readily
reduces to the corresponding one on any fiber, namely the simple
case, which is completely determined in \cite{LLW}:
\begin{equation*}
(\tb_{s_1}^{k_1 - j_1} \cdots\, \tb_{s_n}^{k_n - j_n})^S \langle
h^{j_1}, \ldots, h^{j_n}\rangle_{0, n, d\ell}^{\rm simple} =
(\prod \tb_{s_i}^{k_1 - j_1})^S N_{\vec j}\; d^{n - 3} (-1)^{(d-1)(n+1)}.
\end{equation*}

On the contrary, if the strict inequality holds, by the dimension
counting in the simple case, the restriction of the fiber integral
$\Psi_{n*}(\cdot)$ to points in $S$ vanishes. In fact the fiber
integral is represented by a cycle $S_{\vec j} \subset S$ with
codimension
\[
 \m := \sum j_i - (2r + 1 + n - 3).
\]
The structure of
$S_{\vec j}$ necessarily depends on the bundles $F$ and $F'$.

One would expect the end formula for $\Psi_{n*}(\cdot)$ to be
\begin{equation*}
s_{\m}(F + F'^*)\,N_{\vec j}\;d^{n - 3}
\end{equation*}
with $N_{\vec j} = 1$ for $n \le 3$ so that the difference of the
corresponding generating functions on $X$ and $X'$ cancels out with
the classical defect on cup product. Unfortunately the actual
behavior of these Gromov--Witten invariants with base dimension $s
> 0$ is more delicate than this.

Notice that the new phenomenon does not occur for $n = 2$. In that
case, $k_1 + k_2 = 2r + s$, $j_1 = j_2 = r$ and we may assume that
$\tb_{s_2}$ is running through the dual basis of $\tb_{s_1}$. Since
then the nontrivial terms only appear when $\tb_{s_1}$ and $\tb_{s_2}$
are dual to each other, we get
\begin{align*}
\langle a_1, a_2 \rangle_{0,2,d\ell}^X &= \sum_{s}
(a_1.\tb_s)(a_2.\hatt \tb_s ) \langle h^r, h^r \rangle_d^{\rm
simple} \\
&= (-1)^{(d - 1)(r + 1)}\frac{1}{d} \sum_{s}
(a_1.\tb_s)(a_2.\hatt \tb_s ).
\end{align*}

It is also clear that the new phenomenon does not occur for $P^1$
flops over an arbitrary smooth base $S$. Thus before dealing with
the general cases, we will work out the first (simplest) new case
to demonstrate the general picture that will occur.

\subsection{Twisted relative invariants for $\m = 1$}
Consider $P^r$ flops with $n = 3$ and $j_1 + j_2 + j_3 = (2r + 1) +
1 = 2r + 2$, namely with one more degree (i.e. $\m = 1$) than the
old case. We start with $(j_1, j_2, j_3) = (2, r, r)$. Since
classes from $S$ can be merged into any marked point, the invariant
to be taken care is
\begin{equation*}
\langle h^2, h^r, \bart h^r \rangle_d^X
\end{equation*}
for some $\bart \in H^{2(s - 1)}(S)$. Equivalently we define the fiber
integral
\begin{equation*}
\Big\langle \prod_{i = 1}^n h^{j_i} \Big\rangle_d^{/S} :=
\Psi_{n*} \Big(\prod_{i = 1}^n e_i^* h^{j_i}. e(R^1 ft_*
e_{n + 1}^* N)  \Big) \in A(S)
\end{equation*}
to be a {\em $\bar\psi$-relative invariant over $S$} and we are
computing
\begin{equation*}
\langle h^2, h^r, \bart h^r \rangle_d^X = (\langle h^{2}, h^{r},
h^{r} \rangle_d^{/S}.\bart)^S
\end{equation*}
now. Notice that for $r = 2$, $6 \ge j_1 + j_2 + j_3
> 5$ hence $(2, 2, 2)$ is precisely the only new case to compute.

The basic idea is to use the \emph{divisor relation} \cite{LP} (for $n \ge 3$ points invariants)
\begin{equation} \label{div-rel-1}
e_i^*h = e_j^*h + \sum_{d' + d'' = d}(d'' [D_{ik, d'|j,
d''}]^{virt} - d' [D_{i, d'|jk, d''}]^{virt})
\end{equation}
to move various $h$'s into the same marked point. This type of
process is also referred as \emph{divisorial reconstruction} in this paper.
Once the power exceeds $r$, the Chern polynomial relation reduces
$h^{r + 1}$ into lower degree ones coupled with (Chern) classes
from the base $S$. This will eventually reduce the {\em new
invariants} to {\em old cases}. While this procedure is well known
as the {\em reconstruction principle} in Gromov--Witten theory, the
moral here is to show that this reconstruction transforms perfectly
under flops.

Let $\Delta(X) = \sum_{\mu} T_{\mu} \otimes T^{\mu}$ be a diagonal splitting of
$\Delta(X) \subset X \times X$.
That is, $\{T_{\mu}\}$ is a cohomology basis of $H(X)$ with dual basis
$\{T^{\mu}\}$.
Apply the divisor relation \eqref{div-rel-1}
we get
\begin{align*}
\langle h^2, h^r, \bart h^r \rangle_d &= \langle h, h^{r + 1}, \bart h^r
\rangle_d\\
&\quad + \sum_{d' + d'' = d} \sum_{\mu}d'' \langle h, \bart h^r,
T_{\mu}\rangle_{d'} \langle T^{\mu}, h^r \rangle_{d''} - d' \langle h,
T_{\mu}\rangle_{d'} \langle T^{\mu}, h^r, \bart h^r \rangle_{d''}.
\end{align*}
The last terms vanish since there are no (non-trivial) two point
invariants of the form $\langle h, T_{\mu}\rangle_{d'}$.

Since $h^{r + 1} = -c_1h^r -c_2h^{r - 1} - \cdots - c_{r + 1}$, the
first term clearly equals
\begin{equation*}
-(c_1.\bart )^S \langle h, h^r, h^r \rangle_d^{\rm simple} = -(-1)^{(d
- 1)(r + 1)}(c_1.\bart )^S.
\end{equation*}

For the second terms, notice that the only degree zero invariant is
given by 3-point classical cup product. Hence if $d' = 0$ then we
may select $\{T^{\mu}\}$ in the way that $h.\bart h^r$ appears as one of the
basis elements, say $T^0 = \bart h^{r + 1}$ (this is not part of the
canonical basis). Thus $d'' = d$ and the term equals
\begin{align*}
& d\langle h, \bart h^r, T_0 \rangle_0 \langle \bart h^{r + 1}, h^r\rangle_d
\\
&= -d(c_1.\bart )^S \langle h^r, h^r \rangle_d^{\rm simple} = -(-1)^{(d
- 1)(r + 1)}(c_1.\bart )^S.
\end{align*}

It remains to consider $1 \le d'' \le d - 1$. In this case we may
assume that $T_0 = \hatt \bart   h^r$ since no lower power in $h$ is
allowed. To compute $T^0$ explicitly, since we are considering
extremal rays, we may work on the projective local model $X_{loc} =
P(N_{Z/X} \oplus \mathscr{O})$ of $X$ along $Z$.

By applying Lemma \ref{dual-basis} to $H(X_{loc})$, we get

\begin{lemma} \label{defect-1}
Let $\{z_i\}$ be a basis of $H(Z)$ and $\xi =
c_1(\mathscr{O}_{P(N \oplus \mathscr{O})}(1))$ be the class of the
infinity divisor $E$. The dual basis for $\{z_i\xi^{r + 1 - j}\}_{j
\le r + 1}$ is given by $\{\hatt z_i \Theta_j\}_{j \le r + 1}$ where
\begin{equation*}
\Theta_j := c_j(Q_N) = \xi^j + c_1(N)\xi^{j - 1} + \cdots + c_j(N).
\end{equation*}
In particular, $\Theta_j|_Z = c_j(N)$. Moreover, since $N =
\bar\psi^* F' \otimes \mathscr{O}(-1)$, we have
\begin{equation*}
c_{r + 1}(N) = (-1)^{r + 1}(h^{r + 1} - c_1'h^r + \cdots + (-1)^{r
+ 1}c_{r + 1}').
\end{equation*}
\end{lemma}

Now if $z_0 = \hatt \bart h^r$ and
$T_0 = z_0 \xi^0 = \hatt \bart h^r$, then
$T^0 = \bart \Theta_{r + 1}$ and the invariants become
\begin{align*}
&d" \langle h, \bart h^r, \hatt \bart h^r  \rangle_{d'} \langle \bart c_{r +
1}(N), h^r \rangle_{d"} \\
&= -(-1)^{(d' - 1)(r + 1)}(-1)^{r + 1} d"(\bart .(c_1 + c_1'))^S
\langle h^r, h^r
\rangle_{d"}^{\rm simple} \\
&= -(-1)^{(d' - 1 + d" - 1 + 1)(r + 1)} ((c_1 + c_1').\bart )^S \\
&= -(-1)^{(d - 1)(r + 1)} ((c_1 + c_1').\bart )^S.
\end{align*}
Summing together, we get
\begin{align*}
\langle h^2, h^r, \bart h^r \rangle_d = (-1)^{(d - 1)(r + 1)}
\Big(((-c_1 + c_1').\bart )^S -d((c_1 + c_1').\bart )^S\Big).
\end{align*}

By exactly the same procedure, as long as $j_2 < r$ or $j_3 < r$,
the boundary terms in the divisor relation necessarily vanish by
the exact knowledge on 2-point invariants, hence
\begin{equation*}
\langle h^{j_1}, h^{j_2}, \bart h^{j_3}\rangle_d = \langle h^{j_1 - 1},
h^{j_2 + 1}, \bart h^{j_3}\rangle_d.
\end{equation*}
In particular, any invariant with $j_1 + j_2 + j_3 = 2r + 2$ may be
inductively transformed into $\langle h^2, h^r, \bart h^r\rangle_d$. Hence
we have shown

\begin{proposition}[$n = 3$, $\m = 1$] \label{mu=1}
For $\sum_{i = 1}^3 j_i = 2r + 2$ and $\bart  \in H^{2(s - 1)}(S)$,
\begin{equation*}
\langle h^{j_1}, h^{j_2}, \bart h^{j_3}\rangle_d = (-1)^{(d - 1)(r +
1)}\Big((\tilde s_1.\bart )^S -d (c_1(F + F').\bart )^S\Big).
\end{equation*}
\end{proposition}

As in \cite{LLW}, this implies that the 3-point {\em extremal
quantum corrections} for $X$ and $X'$ remedy the defect of
classical cup product for the cases $\m = 1$.

To see this, it is convenient to consider the basic rational
function
\begin{equation} \label{e:brf}
\f(q) := \frac{q}{1 - (-1)^{r + 1}q} = \sum_{d \ge 1} (-1)^{(d
-1)(r + 1)} q^{d},
\end{equation}
which is the 3-point extremal correction for the case $\m = 0$. It
is clear that
\begin{equation*}
\f(q) + \f(q^{-1}) = (-1)^r.
\end{equation*}

Since $\T(\bart h^j) = (-1)^j \bart h'^j$ for $j \le r$, the geometric series
on $X$
\begin{equation*}
\sum_{d \ge 1} (-1)^{(d -1)(r + 1)} (\tilde s_1.\bart )^S\, q^{d\ell} =
(\tilde s_1.\bart )^S \f(q^\ell)
\end{equation*}
together with its counterpart on $X'$ {\em exactly} correct the
classical term via
\begin{align*}
&(\tilde s_1.\bart )^S \f(q^\ell) - (-1)^{j_1 + j_2 + j_3}
(\tilde s_1'.\bart )^S \f(q^{\ell'}) \\
&= (\tilde s_1.\bart )^S (\f(q^\ell) + \f(q^{-\ell})) = (-1)^r (\tilde
s_1.\bart )^S.
\end{align*}

The new feature for $\m = 1$ is that we also have contributions
involving the differential operator $\delta_h =
q^{\ell}\,d/dq^{\ell}$, namely
\begin{equation*}
-(c_1(F + F').\bart )^S\sum_{d \ge 1} (-1)^{(d -1)(r + 1)} d q^{d\ell}
= -(c_1(F + F').\bart )^S\, \delta_h \f(q^\ell).
\end{equation*}
This higher order series {\em does not} occur as corrections to the
classical defect, though it is still derived from the $\m = 0$
information together with the classical (bundle-theoretic) data. Of
course it is invariant under $P^r$ flops in terms of analytic
continuation.

\begin{remark}\label{caution}
It is helpful to comment on $\bart h^j$ and $\T (\bart h^j)$ to avoid
confusion. Since the Gromov--Witten theory of extremal curve classes
localizes to $Z$, $\bart h^j$ is regarded as $a|_Z$ for some
$a \in H(X)$. If $j \le r$, the familiar formula
$\T a|_{Z'} = (-1)^j\bart h'^{j}$ follows from Proposition \ref{H}, Lemma
\ref{dual-basis} and the invariance of Poincar\'e pairing. However
this formula is not true for $j > r$. Instead, by the Segre
relation $\bar\psi_* h^{r + \m} = s_\m$, we find that $h^{r + \m
} = s_\m h^r + \mbox{(lower order terms)}$. This observation will
be useful later.
\end{remark}

\subsection{Twisted relative invariants for general $\m$}
We will show that when $\sum_{i = 1}^3 j_i = 2r + 1 + \m$ ($\m
\le r - 1$), there is a degree $\m$ cohomology valued polynomial
$W^{F, F'}_{\m}(d) = \sum_{i = 0}^\m w_{\m, i}(F, F')\, d^i$
with coefficients $w_{\m, i}(F,F') \in H^{2\m}(S,\mathbb{Q})$ such that for any class $\bart  \in H^{2(r - \m)}(S)$,
\begin{align*}
\langle h^{j_1}, h^{j_2}, \bart h^{j_3}\rangle_d &= (-1)^{(d - 1)(r +
1)} (W^{F, F'}_{\m}.\bart )^S(d)\\
&:= (-1)^{(d - 1)(r +
1)} \sum_{i = 0}^\m (w_{\m, i}(F, F').\bart )^S\, d^i.
\end{align*}
Hence the 3-point extremal correction is given by
\begin{equation*}
\langle h^{j_1}, h^{j_2}, \bart h^{j_3}\rangle_+ := \sum_{d \ge 1}
\langle h^{j_1}, h^{j_2}, \bart h^{j_3}\rangle_d\, q^{d\ell} = (W^{F,
F'}_{\m}.\bart )^S (\delta_h) \f(q^\ell).
\end{equation*}
and the corresponding $\bar\psi$-relative invariant is equal to
\begin{equation*}
\langle h^{j_1}, h^{j_2}, h^{j_3}\rangle^{/S}_+ = W^{F,F'}_\m(\delta_h) \f(q^\ell).
\end{equation*}

The constant term of $W^{F, F'}_\m$ is the $\m$th Segre class of
$F + F'^*$. This is what we need because (as in the $\m = 1$ case)
\begin{align*}
& \tilde s_\m \f(q^{\ell}) - (-1)^{j_1 + j_2 + j_3} \tilde
s_\m' \f(q^{\ell'})  = (-1)^r \tilde s_\m.
\end{align*}
That is, the classical defect is corrected.

Similarly, for the $d^i$ component with $i \ge 1$,
\begin{equation*}
w_{\m, i}\, \delta_h^i \f(q^\ell) = w_{\m, i}\, (-\delta_{h'})^i
((-1)^r - \f(q^{\ell'})) = (-1)^{i + 1} w_{\m, i}\, \delta_{h'}^i
\f(q^{\ell'}).
\end{equation*}
This is expected to agree with $(-1)^{j_1 + j_2 + j_3} w'_{\m,
i}\,\delta_{h'}^i \f(q^{\ell'})$. Hence we require the alternating
nature of $W$:
\begin{equation*}
w_{\m, i}(F', F) = (-1)^{\m + i} w_{\m, i}(F, F').
\end{equation*}

\begin{remark}
We ignore the degree zero (classical) invariants in the formulation
since they depends on the global geometry of $X$ and $X'$ and could
not be expressed by local universal formula (only their difference
could be).
\end{remark}

Recall that for $1\le \m \le r - 1$, any 3-point invariant
$\langle \bart_1h^{j_1}, \bart_2h^{j_2}, \bart_3h^{j_3}\rangle_d$ with $1\le
j_i \le r$ and $\sum j_i = (2r + 1) + \m$ is equal to the standard
form $\langle h^{\m + 1}, h^r, \bart h^r \rangle_d$ where $\bart  = \bart_1 \bart_2
\bart_3 \in H^{2(s - \m)}(S)$. The study of it is based on the recursive
formula on extremal corrections $W_\m := \langle h^{\m + 1}, h^r,
h^r \rangle_+^{/S}$:

\begin{proposition}\label{recursive}
\begin{equation*}
W_\m = s_\m \f + \sum_{j = 1}^{\m} W_{\m - j}\big((-1)^r c_j \f
- (-1)^{r + j} c_j' \f - c_j\big).
\end{equation*}
\end{proposition}

\begin{proof}
As in \cite{LLW}, by using the operator $\delta_h$, the divisor
relation can be used to obtain splitting relation of generating
series
\begin{align*}
\langle h^{\m + 1}, h^r, \bart h^r \rangle_+ = \langle h^\m, h^{r +
1}, \bart h^r \rangle_+ + \sum_i \langle h^\m, \bart h^r, T_{\mu} \rangle_+
\delta_h \langle T^{\mu}, h^r \rangle_+ + (s_\m.\bart )^S \f.
\end{align*}
The last term is coming from the case with $d_1 = 0$:
\begin{equation*}
\sum_{\mu} \langle h^\m, \bart h^r, T_{\mu} \rangle_0 \delta_h \langle T^{\mu},
h^r \rangle_+ = \delta_h \langle \bart h^{\m + r}, h^r \rangle_+ =
(s_\m.\bart )^S \f.
\end{equation*}
Here the Segre relation $h^{r + \m } = s_\m h^r + \mbox{(lower
order terms)}$ and the complete knowledge of 2-point invariants is
used.

By the Chern polynomial relation, the first term equals
\begin{equation*}
- \sum_{j = 1}^\m \langle h^\m, c_j h^{r + 1 - j}, \bart h^r
\rangle_+ = -\sum_{j = 1}^\m \langle h^{\m - j + 1}, h^r,
c_j\bart h^r \rangle_+ = -\sum_{j = 1}^\m (W_{\m - j}.c_j\bart )^S.
\end{equation*}

For the second sum, we take the degree $r + 1$ part of $T_{\mu}$'s
being of the form $\{ \bart_j h^{r + 1 - j} \}_{j = 1}^{\m}$ with $\bart_j
\in H^{2j}(S)$ to be determined later. Then as in the previous
calculation, using local models, the corresponding dual basis
$T^{\mu}$'s are given by $\{ \hatt \bart_j H_{j - 1} \Theta_{r + 1} \}_{j =
1}^{\m}$. We need the $h^r$ part of
\begin{align*}
&H_{j - 1} \Theta_{r + 1} \\
&= (-1)^{r + 1} (h^{j - 1} + c_1 h^{j - 2} + \cdots + c_{j -
1})(h^{r + 1} - c_1' h^r + \cdots + (-1)^{r + 1} c_{r + 1}')
\end{align*}
in the standard presentation of $H(Z)$. By $\tilde c := c(F +
F'^*) = c(F)c(F'^*)$, it is $(-1)^{r + 1}$ times the $h^r$ part of
\begin{equation*}
h^r(\tilde c_j - c_j) + h^{r + 1}\tilde c_{j - 1} + h^{r + 2}
\tilde c_{j - 2} + \cdots + h^{r + j}.
\end{equation*}
By the Segre relation and $c(F'^*) = s(F)c(F + F'^*)$, the term is
\begin{equation*}
h^r(\tilde c_j + s_1 \tilde c_{j - 1} + s_2 \tilde c_{j - 2} +
\cdots + s_{j - 1}\tilde c_1 + s_j - c_j) = h^r((-1)^j c_j' -
c_j).
\end{equation*}

Now we let $\bart_j = (-1)^j c_j' - c_j$, and then the sum becomes
\begin{equation*}
(-1)^{r + 1} \sum_{j = 1}^\m \langle h^\m, \bart h^r, \bart_j h^{r + 1 -
j} \rangle_+ \f = (-1)^{r + 1} \sum_{j = 1}^\m (W_{\m - j}((-1)^j
c_j' - c_j)\f.\bart )^S.
\end{equation*}
The result follows by putting the three parts together.
\end{proof}

\begin{theorem}[($=$ Theorem \ref{q-correct})]
The $\bar\psi$-relative invariant over $S$
$$W_\m = \langle
h^{j_1}, h^{j_2}, h^{j_3} \rangle_+^{/S}$$
with $1 \le j_i \le r$, $\m = \sum j_i - (2r + 1) \le r - 1$ is
the action on $\f$ by a universal (in $c(F)$ and $c(F')$) rational
cohomology valued polynomial of degree $\m$ in $\delta_h$, which
is independent of the choices of $j_i$'s and satisfies the
functional equation
\begin{equation*}
W_\m - (-1)^{\m + 1}W_\m' = (-1)^r \tilde s_\m
\end{equation*}
for $0 \le \m \le r - 1$.
\end{theorem}

\begin{proof}
Since $W_0 = \f$, by Proposition \ref{recursive}, it is clear that
$W_\m$ is recursively and uniquely determined, which is a degree
$\m + 1$ polynomial in $\f$ with coefficients being universal
polynomial in $c(F)$ and $c(F')$ of pure degree $\m$.

Let 
$$\delta = \delta_h = q d/dq.$$ 
In order to rewrite $W_\m$ as a degree $\m$ polynomial in $\delta \f$, 
we start with the basic relation
\begin{equation*}
\delta \f = \f + (-1)^{r + 1} \f^2.
\end{equation*}
Since $\delta(fg) = (\delta f)g + f\delta g$, it follows
inductively that $\delta^m \f$ can be expressed as $P_m(\f) = \f +
\cdots + (-1)^{m(r + 1)} m! \f^{m + 1}$ with $P_m$ being an integral
universal polynomial of degree $m + 1$. Solving the upper
triangular system between $\delta^m \f$'s and $\f^{m + 1}$'s gives
$\f^{\m + 1} = (-1)^{m(r + 1)}\delta^\m \f /\m ! + \cdots =
Q_\m(\delta) \f$ with $Q_\m$ being a rational polynomial. Clearly
$W_\m$ then admits a corresponding rational cohomology valued
expression as expected.

It remains to check that $W_\m$ satisfies the required functional
equation
\begin{equation*}
W_\m - (-1)^{\m + 1}W_\m' = (-1)^r \tilde s_\m.
\end{equation*}
We will prove it by induction. The case $\m = 0$ goes back to $\f +
\f' = (-1)^r$ where $\f := \f(q^\ell)$ and $\f' := \f(q^{\ell'}) \equiv
\f(q^{-\ell})$ under the correspondence $\T$.

Assume the functional equation holds for all $j < \m$. Then
\begin{equation*}
W_\m = s_\m \f + \sum_{j = 1}^{\m} W_{\m - j}\big((-1)^r c_j \f
- (-1)^{r + j} c_j' \f - c_j\big),
\end{equation*}
\begin{equation*}
W_\m' = s_\m' \f' + \sum_{j = 1}^{\m} W_{\m - j}'\big((-1)^r
c_j' \f' - (-1)^{r + j} c_j \f' - c_j'\big).
\end{equation*}
By substituting
\begin{equation*}
W_{\m - j}' = (-1)^{\m - j + 1} W_{\m - j} + (-1)^{r + \m - j}
\tilde s_{\m - j}
\end{equation*}
into $W_\m'$, we compute, after cancellations,
\begin{align*}
&W_\m - (-1)^{\m + 1} W_\m' \\
&= s_\m \f + (-1)^\m s_\m' \f' + \sum_{j = 1}^\m \big((-1)^j
\tilde s_{\m - j} c'_j \f' - \tilde s_{\m - j} c_j \f' - (-1)^{r -
j}
\tilde s_{\m - j} c_j'\big) \\
&= s_\m \f + (-1)^\m s_\m' \f' + (s_\m - \tilde s_\m)\f' -
((-1)^\m s_\m' - \tilde s_\m)\f' - (-1)^r (s_\m - \tilde s_\m) \\
&= s_\m(\f + \f') - (-1)^r s_\m  + (-1)^r \tilde s_\m \\
&= (-1)^r \tilde s_\m,
\end{align*}
where both directions of the Whitney sum relations
\begin{equation*}
s(F) = s(F + F'^*)c(F'^*); \qquad s(F'^*) = s(F + F'^*)c(F)
\end{equation*}
are used. The proof is completed.
\end{proof}

\begin{corollary}\label{3-pt-inv}
For any ordinary flop over a smooth base, we have
\begin{equation*}
\T \langle a_1, a_2, a_3 \rangle^X \cong \langle
\T a_1, \T a_2, \T a_3 \rangle^{X'}
\end{equation*}
modulo non-extremal curve classes.
\end{corollary}

\subsection{Functional equations for $n \ge 3$ point extremal
functions}

For ordinary flops over any smooth base, we will show that
Corollary \ref{3-pt-inv} extends to all $n \ge 4$. Namely
\begin{equation*}
\T \langle a_1, \cdots, a_n \rangle^X \cong \langle
\T a_1, \cdots, \T a_n \rangle^{X'}
\end{equation*}
modulo non-extremal curve classes.

By restricting to $Z$ and $Z'$, it is equivalent to the nice
looking formula
\begin{equation*}
\T\langle h^{j_1}, \cdots, \bart h^{j_n} \rangle \cong (-1)^{\sum
j_i}\langle h'^{j_1}, \cdots, \bart h'^{j_n} \rangle
\end{equation*}
for all $1 \le j_l \le r$, where for notational simplicity the
$n$-point functions in this section refer to {\em extremal
functions}, that is, the sum is only over $\mathbb{Z}_+\ell$.

Notices that $\T(\bart h^j) = (-1)^j\bart h'^j$ only for $j \le r$ and it
fails in general for $j > r$ if the base $S$ is non-trivial. In
fact, we have

\begin{lemma} \label{defect-2}
\begin{equation*}
\T (h^{r + 1}) - (\T h)^{r + 1} = (-1)^{r + 1} \T \Theta_{r + 1}
\end{equation*}
along $Z'$
\end{lemma}

\begin{proof}
This is simply a reformulation of Lemma \ref{defect-1}.
\end{proof}

It is easy to see that $\T\langle h^{j_1}, \cdots, \bart h^{j_n} \rangle
\not\cong (-1)^{\sum j_i}\langle h'^{j_1}, \cdots, \bart h'^{j_n}
\rangle$ if some $j_l > r$. This appears as the subtle point in
proving the functional equations for $n \ge 4$ points. The above
lemma plays a crucial role in analyzing this.

\begin{theorem}\label{big-q}
Let $f: X \dasharrow X'$ be an ordinary $P^r$ flop with exceptional
loci $Z = P(F) \to S$ and $Z' = P(F') \to S$. Then for $n \ge 3$,
\begin{equation*}
\T\langle h^{j_1}, \cdots, \bart h^{j_n} \rangle^X \cong \langle \T
h^{j_1}, \cdots, \T \bart h^{j_n} \rangle^{X'}
\end{equation*}
for all $j_l$'s and $\bart  \in H^{2(s - \m)}(S)$ with $\m = \sum_{l =
1}^n j_l - (2r + 1 + n - 3)$.
\end{theorem}

\begin{proof}
This holds for $n = 3$ by Corollary \ref{3-pt-inv}. Suppose this
has been proven up to some $n \ge 3$. The basic idea is that an
iterated application of the divisor relation using the operator
$\delta_h$ should allow us to reduce an $n + 1$ point extremal
function to ones with fewer marked points. The technical details
however should be traced carefully.

The first point to make is on the diagonal splitting $\Delta(X) =
\sum T_{\mu} \otimes T^{\mu}$. Since the Poincar\'e pairing is preserved,
$\T T^{\mu}$ is still the dual basis of $\T T_{\mu}$ in $H(X')$. Thus we
may take the diagonal splitting on the $X'$ side to be $\Delta(X')
= \sum \T T_{\mu} \otimes \T T^{\mu}$.

We only need to prove the case that all $j_l \le r$. The $P^1$
flops always have $\m = 0$ and the proof is reduced to the simple
case. So we assume that $r \ge 2$.

We will prove the functional equation by further induction on
$j_1$. The case $j_1 = 1$ holds by the divisor axiom and induction,
so we assume that $j_1 \ge 2$. By applying the divisor relation to
$(i,j, k) = (1, 2, 3)$, we get
\begin{align*}
& \langle h^{j_1}, h^{j_2}, h^{j_3}, \cdots \rangle = \langle
h^{j_1 - 1}, h^{j_2 + 1}, h^{j_3}, \cdots \rangle  \\
& + \sum_{\mu} \langle h^{j_1 - 1}, h^{j_3}, \cdots, T_{\mu} \rangle
\delta_h \langle h^{j_2}, \cdots, T^{\mu} \rangle - \delta_h \langle
h^{j_1 - 1}, \cdots, T_{\mu} \rangle \langle h^{j_2}, h^{j_3}, \cdots,
T^{\mu} \rangle.
\end{align*}

Since $j_1 - 1 < r$, $\langle h^{j_1 - 1}, \cdots, T_{\mu} \rangle$ can
not be a 2-point invariant unless it is trivial. Hence we may
assume that $\langle h^{j_2}, h^{j_3}, \cdots, T^{\mu} \rangle$ has
fewer points.

The term $\langle h^{j_1 - 1}, h^{j_2+1}, h^{j_3},
\cdots \rangle$ is also handled by induction since $j_1 - 1 < j_1$.
Thus we may apply $\T$ to the equation and apply induction to get
\begin{align*}
\T\langle h^{j_1}, h^{j_2}, h^{j_3}, \cdots \rangle &= \langle
\T h^{j_1 - 1}, \T h^{j_2 + 1}, \T h^{j_3}, \cdots \rangle  \\
& \qquad + \sum_{\mu} \langle \T h^{j_1 - 1}, \T h^{j_3}, \cdots, \T
T_{\mu} \rangle \delta_{\T h} \T \langle h^{j_2}, \cdots, T^{\mu} \rangle
\\
&\qquad \qquad - \delta_{\T h} \langle \T h^{j_1 - 1}, \cdots, \T
T_{\mu} \rangle \langle \T h^{j_2}, \T h^{j_3}, \cdots, \T T^{\mu}
\rangle,
\end{align*}
where $\T \circ \delta_h = \delta_{\T h} \circ \T$ by \cite{LLW},
Lemma 5.5.

Notice that in the first summand,
\begin{equation*}
\T \langle h^{j_2}, \cdots, T^{\mu} \rangle = \langle \T h^{j_2},
\cdots, \T T^{\mu} \rangle
\end{equation*}
if it is not a 2-point invariant. Also the 2-point case survives
precisely when $j_2 = r$ and $T^{\mu} = \mbox{pt}.h^r$. In that case,
by the invariance of 3-point extremal functions in the $\m = 0$
(simple) case, the corresponding term becomes
\begin{align*}
\T \delta_h \langle h^r, T^{\mu} \rangle &= \T \langle h, h^r , T^{\mu}
\rangle_+ \\
&= \langle \T h, \T h^r, \T T^{\mu} \rangle_+ + (-1)^r = \delta_{\T h}
\langle \T h^r, \T T^{\mu} \rangle + (-1)^r.
\end{align*}
Also $T_{\mu}|_Z = \Theta_{r + 1}|_Z$. Hence by Lemma \ref{defect-2}
the extra $(-1)^r$ contributes
\begin{equation*}
-\langle \T h^{j_1 - 1}, \T h^{j_3}, \cdots, \T h^{r + 1}\rangle -
\langle \T h^{j_1 - 1}, \T h^{j_3}, \cdots, (\T h)^{r + 1}\rangle.
\end{equation*}
Since $j_2 = r$, the LHS cancels with the first term in the divisor
relation and we end up with the RHS as the main term.

Now we compare it with the similar divisor relation for
\begin{equation*}
\langle \T h^{j_1}, \T h^{j_2}, \T h^{j_3}, \cdots \rangle =
\langle \T h.\T h^{j_1 - 1}, \T h^{j_2}, \T h^{j_3}, \cdots
\rangle
\end{equation*}
under the diagonal splitting $\Delta(X') = \sum_{\mu} \T T_{\mu} \otimes \T
T^{\mu}$. Namely
\begin{align*}
 &\langle \T h^{j_1}, \T h^{j_2}, \T h^{j_3}, \cdots \rangle \\
 &= \langle
\T h^{j_1 - 1}, \T h.\T h^{j_2}, \T h^{j_3}, \cdots \rangle  \\
&\quad + \sum_{\mu} \langle \T h^{j_1 - 1}, \T h^{j_3}, \cdots, \T
T_{\mu} \rangle \delta_{\T h} \langle \T h^{j_2}, \cdots, \T T^{\mu}
\rangle \\
 &\qquad - \delta_{\T h} \langle \T h^{j_1 - 1}, \cdots, \T
T_{\mu} \rangle \langle \T h^{j_2}, \T h^{j_3}, \cdots, \T T^{\mu}
\rangle.
\end{align*}

If $j_2 < r$ then there is no 2-point splitting and $\T h.\T
h^{j_2} = \T h^{j_2 + 1}$, hence the functional equation holds. If
$j_2 = r$ then $\T h.\T h^r = (\T h)^{r + 1}$. This again agrees
with the main term obtained above. Hence the proof of functional
equations is complete by induction.
\end{proof}

Formula for $W_{\vec j} := \langle h^{j_1}, \cdots, h^{j_n}
\rangle^{/S}$ can be achieved by a similar process as in Lemma
\ref{recursive}, whose exact form would not be pursued here. In
general it depends on the vector $\vec j$ instead of $\sum j_i$.

\begin{remark}
Theorem \ref{q-correct} and \ref{big-q} (for the special case $F' = F^*$) have been applied in \cite{FW} to study stratified Mukai flops. In particular they provide non-trivial quantum corrections to flops of type $A_{n, 2}$, $D_5$ and $E_{6, I}$.
\end{remark}

\section{Degeneration analysis revisited} \label{degeneration}

Our next task is to \emph{compare} the Gromov--Witten invariants of
$X$ and $X'$ for all genera and for curve classes other than the
flopped curve. As in \cite{LLW}, we use the degeneration formula
\cite{LiRu, Li} to reduce the problem to local models. This
has been achieved for \emph{simple} ordinary flops in \cite{LLW}
for \emph{genus zero} invariants. In this section we extend the
argument to the general case and establish Theorem \ref{deg-red}
($=$ Proposition \ref{reduction} $+$ \ref{p:4.8}) in the
introduction.

\subsection{The degeneration formula} We start by reviewing the
basic setup. Details can be found in the above references.

Consider a pair $(Y, E)$ with $E \hookrightarrow Y$ a smooth
divisor. Let $\Gamma = (g, n, \beta, \rho, \mu)$ with $\mu =
(\mu_1, \ldots, \mu_\rho) \in \mathbb{N}^\rho$ a partition of the
intersection number $(\beta.E)= |\mu| := \sum_{i = 1}^\rho \mu_i$.
For $A \in H(Y)^{\otimes n}$ and $\s \in H(E)^{\otimes \rho}$,
the relative invariant of stable maps with topological type
$\Gamma$ (i.e.~with contact order $\mu_i$ in $E$ at the $i$-th
contact point) is
\begin{equation*}
\langle A \mid \s, \mu \rangle^{(Y, E)}_{\Gamma} :=
\int_{[\overline{M}_{\Gamma}(Y, E)]^{virt}} e_Y^* A \cup e_E^* \s
\end{equation*}
where $e_Y: \overline{M}_{\Gamma}(Y, E) \to Y^n$, $e_E:
\overline{M}_{\Gamma}(Y, E) \to E^\rho$ are evaluation maps on
marked points and contact points respectively. If $\Gamma =
\coprod_\pi \Gamma^\pi$, the relative invariant with disconnected
domain curve is defined by the product rule:
\begin{equation*}
\langle A \mid \s, \mu \rangle^{\bullet (Y, E)}_{\Gamma} :=
\prod_\pi \langle A \mid \s, \mu \rangle^{(Y, E)}_{\Gamma^\pi}.
\end{equation*}

We apply the degeneration formula to the following situation. Let
$X$ be a smooth variety and $Z \subset X$ be a smooth subvariety.
Let $\Phi:W \to \mathscr{X}$ be its \emph{degeneration to the
normal cone}, the blow-up of $X \times \mathbb{A}^1$ along $Z
\times \{0\}$. Let $t\in\mathbb{A}^1$. Then $W_t \cong X$ for all
$t\ne 0$ and $W_0 = Y_1 \cup Y_2$ with
\begin{equation*}
\phi =\Phi|_{Y_1}: Y_1 \to X
\end{equation*}
the blow-up along $Z$ and
\begin{equation*}
p = \Phi|_{Y_2}: Y_2 := P(N_{Z/X}\oplus \mathscr{O})
\to Z\subset X
\end{equation*}
the projective completion of the normal bundle. $Y_1 \cap Y_2 =: E
= P(N_{Z/X})$ is the $\phi$-exceptional divisor which consists of
the infinity part.

The family $W \to \mathbb{A}^1$ is a degeneration of a trivial
family, so all cohomology classes $\alpha\in
H(X,\mathbb{Z})^{\oplus n}$ have global liftings and the
restriction $\alpha(t)$ on $W_t$ is defined for all $t$. Let $j_i:
Y_i\hookrightarrow W_0$ be the inclusion maps for $i = 1, 2$. Let
$\{\e_i\}$ be a basis of $H(E)$ with $\{\e^i\}$ its dual basis.
$\{\e_I\}$ forms a basis of $H(E^\rho)$ with dual basis
$\{\e^I\}$ where $|I| = \rho$, $\e_I = \e_{i_1}\otimes \cdots
\otimes \e_{i_\rho}$. The \emph{degeneration formula} expresses the
absolute invariants of $X$ in terms of the relative invariants of
the two smooth pairs $(Y_1,E)$ and $(Y_2, E)$:
\begin{equation*}
\langle\alpha\rangle_{g,n,\beta}^X = \sum_{I} \sum_{\eta\in
\Omega_\beta} C_\eta \left.\Big\langle j_1^*\alpha(0) \,\right|\,
\e_I, \mu \Big\rangle_{\Gamma_1}^{\bullet (Y_1,E)}
\left.\Big\langle j_2^*\alpha(0) \,\right|\, \e^I, \mu
\Big\rangle_{\Gamma_2}^{\bullet (Y_2,E)}.
\end{equation*}
Here $\eta = (\Gamma_1, \Gamma_2, I_\rho)$ is an {\it admissible
triple} which consists of (possibly disconnected) topological types
\begin{equation*}
\Gamma_i = \coprod\nolimits_{\pi = 1}^{|\Gamma_i|} \Gamma_i^\pi
\end{equation*}
with the same partition $\mu$ of contact order under the
identification $I_\rho$ of contact points. The gluing $\Gamma_1
+_{I_\rho} \Gamma_2$ has type $(g, n, \beta)$ and is connected. In
particular, $\rho = 0$ if and only if that one of the $\Gamma_i$ is
empty. The total genus $g_i$, total number of marked points $n_i$
and the total degree $\beta_i \in NE(Y_i)$ satisfy the splitting
relations
\begin{equation*}
\begin{split}
g - 1 &= \sum\nolimits_{\pi = 1}^{|\Gamma_1|} (g_1(\pi) - 1)
+ \sum\nolimits_{\pi = 1}^{|\Gamma_2|} (g_2(\pi) - 1) + \rho\\
&= g_1 + g_2 -|\Gamma_1| -|\Gamma_2| + \rho, \\
n &= n_1 + n_2, \\
\beta &= \phi_*\beta_1 + p_*\beta_2.
\end{split}
\end{equation*}
(The first one is the arithmetic genus relation for nodal curves.)

The constants $C_\eta = m(\mu)/|{\rm Aut}\,\eta|$, where $m(\mu) =
\prod \mu_i$ and ${\rm Aut}\,\eta = \{\, \sigma \in S_\rho \mid
\eta^\sigma = \eta \,\}$. We denote by $\Omega$ the set of
equivalence classes of all admissible triples; by $\Omega_\beta$
and $\Omega_\mu$ the subset with fixed degree $\beta$ and fixed
contact order $\mu$ respectively.

Given an ordinary flop $f : X \dashrightarrow X'$, we apply
degeneration to the normal cone to both $X$ and $X'$. Then $Y_1
\cong Y'_1$ and $E = E'$ by the definition of ordinary flops. The
following notations will be used
\begin{equation*}
 Y := {\rm Bl}_Z X \cong Y_1 \cong Y'_1, \quad
 \tilde{E} :=  P(N_{Z/X}\oplus \mathscr{O}), \quad
 \tilde{E}' :=  P(N_{Z'/X'}\oplus \mathscr{O}).
\end{equation*}

Next we discuss the presentation of $\alpha(0)$. Denote by $\iota_1
\equiv j: E \hookrightarrow Y_1 = Y$ and $\iota_2: E
\hookrightarrow Y_2 = \tilde E$ the natural inclusions. The class
$\alpha(0)$ can be represented by $(j_1^*\alpha(0), j_2^*\alpha(0))
= (\alpha_1, \alpha_2)$ with $\alpha_i \in H(Y_i)$ such that
\begin{equation*}
\iota_1^*\alpha_1 = \iota_2^*\alpha_2 \quad\mbox{and}\quad
\phi_*\alpha_1 + p_*\alpha_2 = \alpha.
\end{equation*}
Such representatives are called {\it liftings}, which are not
unique.

The standard choice of lifting is
\begin{equation*}
\alpha_1 = \phi^*\alpha \quad \text{and} \quad \alpha_2 =
p^*(\alpha|_{Z}).
\end{equation*}
Other liftings can be obtained from the standard one by the
following way.

\begin{lemma}[\cite{LLW}]\label{switch}
Let $\alpha(0) = (\alpha_1, \alpha_2)$ be a choice of lifting. Then
\begin{equation*}
\alpha(0) = (\alpha_1 - \iota_{1*} e, \alpha_2 + \iota_{2*} e)
\end{equation*}
is also a lifting for any class $e$ in $E$ of the same dimension as
$\alpha$. Moreover, any two liftings are related in this manner.
\end{lemma}

For an ordinary flop $f:X\dashrightarrow X'$, we compare the
degeneration expressions of $X$ and $X'$. For a given admissible
triple $\eta = (\Gamma_1, \Gamma_2, I_\rho)$ on the degeneration of
$X$, one may pick the corresponding $\eta' = (\Gamma_1', \Gamma_2',
I'_\rho)$ on the degeneration of $X'$ such that $\Gamma_1 =
\Gamma_1'$. Since
\begin{equation*}
\phi^*\alpha - \phi'^*\T \alpha \in \iota_{1*} H(E) \subset
H(Y),
\end{equation*}
Lemma~\ref{switch} implies that one can choose
$\alpha_1=\alpha'_1$. This procedure identifies relative invariants
on the $Y_1 = Y = Y'_1$ from both sides, and we are left with the
comparison of the corresponding relative invariants on $\tilde E$
and $\tilde E'$.

The ordinary flop $f$ induces an ordinary flop
\begin{equation*}
\tilde f : \tilde E \dashrightarrow \tilde E'
\end{equation*}
on the local model. Denote again by $\T$ the cohomology
correspondence induced by the graph closure. Then

\begin{lemma}[\cite{LLW}]\label{coh-red}
Let $f:X\dashrightarrow X'$ be an ordinary flop. Let $\alpha \in
H(X)$ with liftings $\alpha(0) = (\alpha_1, \alpha_2)$ and
$\T\alpha(0) = (\alpha'_1, \alpha'_2)$. Then
\begin{equation*}
\alpha_1 = \alpha'_1 \quad \Longleftrightarrow \quad \T\alpha_2 =
\alpha'_2.
\end{equation*}
\end{lemma}

Now we are in a position to apply the degeneration formula to
reduce the problem to relative invariants of local models.

Notice that $A_1(\tilde E) = \iota_{2*} A_1(E)$ since both are
projective bundles over $Z$. We then have
\begin{equation*}
\phi^*\beta = \beta_1 + \beta_2
\end{equation*}
by regarding $\beta_2$ as a class in $E \subset Y$ (c.f.\
\cite{LLW}).

Define the generating series for genus $g$ (connected) invariants
\begin{equation*}
\langle A \mid \s, \mu \rangle^{(\tilde{E}, E)}_g
 := \sum_{\beta_2 \in NE(\tilde E)}
 \frac{1}{|{\rm Aut}\,\mu|} \langle A \mid \s, \mu
 \rangle_{g,\beta_2}^{(\tilde E, E)}\,q^{\beta_2}.
\end{equation*}
and the similar one with possibly disconnected domain curves
\begin{equation*}
\langle A \mid \s, \mu \rangle^{\bullet (\tilde E, E)} :=
\sum_{\Gamma;\, \mu_\Gamma = \mu} \frac{1}{|{\rm Aut}\,\Gamma|}
\langle A \mid \s, \mu \rangle_{\Gamma}^{\bullet (\tilde E,
E)}\,q^{\beta^\Gamma}\,\kappa^{g^\Gamma - |\Gamma|}.
\end{equation*}

For connected invariants of genus $g$ we assign the $\kappa$-weight
$\kappa^{g - 1}$, while for disconnected ones we simply assign the
product weights.

\begin{proposition}\label{reduction}
To prove $\T \langle \alpha \rangle^X_g \cong \langle \T\alpha
\rangle^{X'}_g$ for all $\alpha$ up to genus $g \le g_0$, it is
enough to show that
\begin{equation*}
\T\langle A \mid \s, \mu \rangle^{(\tilde{E}, E)}_g \cong \langle
\T A \mid \s, \mu \rangle^{(\tilde{E}', E)}_g
\end{equation*}
for all $A, \s, \mu$ and $g \le g_0$.
\end{proposition}

\begin{proof} For the $n$-point function
\begin{equation*}
\langle \alpha \rangle^X = \sum_g \langle \alpha \rangle^X_g
\,\kappa^{g - 1} = \sum_{g;\,\beta \in NE(X)} \langle \alpha
\rangle^X_{g,\beta}\, q^\beta\,\kappa^{g - 1},
\end{equation*}
the degeneration formula gives
\begin{align*}
\langle \alpha \rangle^X &= \sum_{g;\beta \in NE(X)}\sum_{\eta \in
\Omega_\beta} \sum_I C_\eta \langle \alpha_1 \mid \e_I, \mu
\rangle_{\Gamma_1}^{\bullet (Y_1,E)} \langle \alpha_2 \mid \e^I,
\mu \rangle_{\Gamma_2}^{\bullet (Y_2,E)}\,q^{\phi^*\beta}\,
\kappa^{g - 1}\\
&= \sum_{\mu} \sum_I \sum_{\eta \in \Omega_\mu} C_\eta \times \\
&\quad \left(\langle \alpha_1 \mid \e_I, \mu
\rangle_{\Gamma_1}^{\bullet
(Y_1,E)}\,q^{\beta_1}\,\kappa^{g^{\Gamma_1} - |\Gamma_1|}\right)
\left(\langle \alpha_2 \mid \e^I, \mu \rangle_{\Gamma_2}^{\bullet
(Y_2,E)}\,q^{\beta_2}\,\kappa^{g^{\Gamma_2} - |\Gamma_2|}\right)
\kappa^\rho.
\end{align*}
(Notice that $\rho$ is determined by $\mu$.) In this formula, the
variable $q^{\beta_1}$ on $Y_1$ (resp.~$q^{\beta_2}$ on $Y_2$) is
identified with $q^{\phi_* \beta_1}$ (resp.~$q^{p_* \beta_2}$) on
$X$.

To simplify the generating series, we consider also absolute
invariants $\langle \alpha \rangle^{\bullet X}$ with possibly
disconnected domain curves as in the relative case (with product
weights in $\kappa$). Then by comparing the order of automorphisms,
\begin{equation*}
\langle \alpha \rangle^{\bullet X} = \sum_{\mu} m(\mu) \sum_I
\langle \alpha_1 \mid \e_I, \mu \rangle^{\bullet (Y_1,E)} \langle
\alpha_2 \mid \e^I, \mu \rangle^{\bullet (Y_2,E)}\,\kappa^\rho.
\end{equation*}

To compare $\T \langle \alpha \rangle^{\bullet X}$ and $\langle \T
\alpha \rangle^{\bullet X'}$, by Lemma~\ref{coh-red} we may assume
that $\alpha_1 = \alpha_1'$ and $\alpha_2' = \T\alpha_2$. This
choice of cohomology liftings identifies the relative invariants of
$(Y_1, E)$ and those of $(Y'_1, E)$ with the same topological
types. It remains to compare (c.f.~Remark \ref{contact-weight}
below)
\begin{equation*}
\langle \alpha_2 \mid \e^I, \mu \rangle^{\bullet (\tilde E, E)}
\quad \mbox{and} \quad \langle \T\alpha_2 \mid \e^I, \mu
\rangle^{\bullet (\tilde E', E)}.
\end{equation*}

We further split the sum into connected invariants. Let $\Gamma^\n$
be a connected part with the contact order $\mu^\n$ induced from
$\mu$. Denote $P: \mu = \sum_{\n \in P} \mu^\n$ a partition of
$\mu$ and $P(\mu)$ the set of all such partitions. Then
\begin{equation*}
\langle A \mid \s, \mu \rangle^{\bullet (\tilde E, E)} = \sum_{P
\in P(\mu)} \prod_{\n \in P} \sum_{\Gamma^\n} \frac{1}{|{\rm
Aut}\,\mu^\n|} \langle A^\n \mid \s^\n, \mu^\n
\rangle_{\Gamma^\n}^{(\tilde E, E)}\,q^{\beta^{\Gamma^\n}}\,
\kappa^{g^{\Gamma^\pi} - 1}.
\end{equation*}

In the summation over $\Gamma^\pi$, the only index to be summed
over is $\beta^{\Gamma^\n}$ on $\tilde{E}$ and the genus. This
reduces the problem to $\langle A^\n \mid \s^\n, \mu^\n
\rangle^{(\tilde{E}, E)}_g$.

Instead of working with all genera, the proposition follows from
the same argument by reduction modulo $\kappa^{g_0}$.
\end{proof}

\begin{remark} \label{contact-weight}
Notice that there is natural compatibility on our identifications
of the curve classes which keeps track on the contact weight
$|\mu|$. Namely, the identity $\langle \alpha_1 \mid \e_I, \mu
\rangle^{\bullet (Y_1,E)} = \langle \alpha_1 \mid \e_I, \mu
\rangle^{\bullet (Y_1',E)}$ leads to
\begin{equation*}
\T \phi_* \langle \alpha_1 \mid \e_I, \mu \rangle^{\bullet
(Y_1,E)} = q^{|\mu| \ell'} \phi'_* \langle \alpha_1 \mid \e_I, \mu
\rangle^{\bullet (Y_1',E)},
\end{equation*}
while $\T \langle \alpha_2 \mid \e^I, \mu \rangle^{\bullet (\tilde
E, E)} \cong \langle \T\alpha_2 \mid \e^I, \mu \rangle^{\bullet
(\tilde E', E)}$ leads to
\begin{equation*}
\T p_* \langle \alpha_2 \mid \e^I, \mu \rangle^{\bullet (\tilde E,
E)} \cong q^{-|\mu|\ell'} p'_* \langle \T\alpha_2 \mid \e^I, \mu
\rangle^{\bullet (\tilde E', E)}.
\end{equation*}
Thus we may ignore the issue of contact weights in our discussion.
\end{remark}

\subsection{Relative local back to absolute local} \label{ReToAbs}
Now let $X=\tilde{E}$. We shall further reduce the relative cases
to the absolute cases with at most descendent insertions along $E$.
This has been done in \cite{LLW} for genus zero invariants under
simple flops. Here we extend the argument to ordinary flops over
any smooth base $S$ and to all genera.

The local model
\begin{equation*}
\bar p:= \bar\psi \circ p: \tilde E \mathop{\to}^p Z
\mathop{\to}^{\bar \psi} S
\end{equation*}
as well as the flop $f: \tilde E \dasharrow \tilde E'$ are all over
$S$, with each fiber isomorphic to the simple case. Thus the map on
numerical one cycles
\begin{equation*}
\bar p_*: N_1(\tilde E) \to N_1(S)
\end{equation*}
has kernel spanned by the $p$-fiber line class $\gamma$ and $\bar
\psi$-fiber line class $\ell$, which is the flopping log-extremal
ray.

Notice that for general $S$ the structure of $NE(Z)$ could be
complicated and $NE(\tilde E)$ is in general larger than $i_* NE(Z)
\oplus \mathbb{Z}^+\gamma$. For $\beta = \beta_Z + d_2(\beta)
\gamma \in NE(\tilde E)$, while $\beta_Z = p_*\beta$ is necessarily
effective, $d_2(\beta)$ could possibly be negative if (and only if)
$\beta_Z \ne 0$. Nevertheless we have the following:

\begin{lemma}
The correspondence $\T$ is compatible with $N_1(S)$. Namely
\begin{equation*}
\xymatrix{ N_1(\tilde E) \ar[rr]^\T \ar[rd]_{\bar p_* \oplus d_2}
& &
N_1(\tilde E') \ar[ld]^{\bar p'_* \oplus d_2'} \\
& N_1(S) \oplus \mathbb{Z} &}
\end{equation*}
is commutative.
\end{lemma}

\begin{proof}
Since $N_1(\tilde E) = i_* N_1(Z) \oplus \mathbb{Z}\gamma$ and $\T
\gamma = \gamma' + \ell'$, we see that $d_2 = d_2' \circ \T$ and it
is enough to consider $\beta \in N_1(Z)$. Also $\T \ell = -\ell'$,
so the remaining cases are of the form $\beta =
\bar\psi^*\beta_S.H_r$ for $\beta_S \in N_1(S)$. Then $\T \beta =
\bar\psi'^*\beta_S.H'_r$ and it is clear that both $\beta$ and $\T
\beta$ project to $\beta_S$.
\end{proof}

This leads to the following key observation, which applies to both
absolute and relative invariants:

\begin{proposition}
Functional equation of a generating series $\langle A \rangle$ over
Mori cone on local models $f: \tilde E \dasharrow \tilde E'$ is
equivalent to functional equations of its various subseries (fiber
series) $\langle A \rangle_{\beta_S, d_2}$ labeled by $NE(S)
\oplus \mathbb{Z}$. The fiber series is a sum over the affine ray
$\beta \in (d_2 \gamma + \bar\psi^*\beta_S.H_r + \mathbb{Z} \ell)
\cap NE(\tilde E)$.
\end{proposition}

To analyze these fiber series $\langle A \rangle_{\beta_S, d_2}$
with $(\beta_S, d_2) \in NE(S) \oplus \mathbb{Z}$, we consider
the partial order of effectivity (weight) of the quotient Mori cone
\begin{equation*}
W := NE(\tilde E)/\sim, \quad \mbox{$a \sim b$ if and only if
$a - b \in \mathbb{Z} \ell$}.
\end{equation*}
Notice that $a > b$ and $b > a$ lead to $a \sim b$ since $\ell$ is
an extremal ray. Under the natural identification, $W$ can be regarded as a subset of $NE(S) \oplus \mathbb{Z}$. This is \emph{not} the
lexicographical (partial) order on $NE(S) \oplus \mathbb{Z}$, though both notions are all used in our discussions. For ease of notations we also use
\begin{equation*}
[\beta] \equiv (\beta_S, d_2) := (\bar p_*(\beta), d_2(\beta)) \in
W
\end{equation*}
to denote the class of $\beta$ modulo extremal rays.

Given insertions
$$A = (a_1, \ldots, a_n) \in H(\tilde E)^{\oplus n}$$
and weighted partition
$$(\s, \mu) = \{(\s_1, \mu_1), \ldots, (\s_\rho, \mu_\rho)\},$$
the genus $g$ relative invariant
$\langle A \mid \s, \mu \rangle_g$ is summing over classes $\beta =
\beta_Z + d_2 \gamma \in NE(\tilde E)$ with
\begin{equation*}
\sum_{j = 1}^n \deg a_j + \sum_{j = 1}^{\rho} \deg \s_j =
(c_1(\tilde E).\beta) + (\dim \tilde E - 3)(1 - g) + n + \rho -
|\mu|.
\end{equation*}
In this case, $d_2 = (E.\beta) = |\mu|$ is already fixed and
non-negative.

\begin{proposition} \label{p:4.8}
For an ordinary flop $\tilde E \dashrightarrow \tilde E'$, to prove
\begin{equation*}
\T\langle A \mid \s, \mu \rangle_{g, \beta_S} \cong \langle \T A
\mid \s, \mu \rangle_{g, \beta_S}
\end{equation*}
for any $A \in H(\tilde E)^{\oplus n}$, $\beta_S \in NE(S)$ and $(\s, \mu)$ up to genus $g \le
g_0$, it is enough to prove the $\T$-invariance for descendent invariants of $f$-special type. Namely,
\begin{equation*}
\T\langle A, \tau_{k_1} \s_1, \cdots, \tau_{k_\rho} \s_\rho
\rangle^{\tilde{E}}_{g, \beta_S, d_2} \cong \langle \T A,
\tau_{k_1} \s_1, \cdots, \tau_{k_\rho} \s_\rho
\rangle^{\tilde{E}'}_{g, \beta_S, d_2}
\end{equation*}
for any $A \in H(\tilde E)^{\oplus n}$, $k_j \in \mathbb{N}\cup
\{0\}$, $\s_j \in H(E)$ and $\beta_S \in NE(S)$, $d_2 \ge 0$ up
to genus $g \le g_0$.
\end{proposition}

\begin{proof}
The proof proceeds inductively on the 5-tuple
\begin{equation*}
(g, \beta_S, |\mu| = d_2, n, \rho)
\end{equation*}
in the lexicographical order, with $\rho$ in the reverse order.

Given $\langle a_1, \cdots, a_n \mid \s, \mu\rangle_{g,
\beta_S}$, since $\rho \le |\mu|$, there are only finitely many
5-tuples of lower order. The proposition holds for those cases by
the induction hypothesis.

We apply degeneration to the normal cone for $Z \hookrightarrow
\tilde E$ to get $W \to \mathbb{A}^1$. Then $W_0 = Y_1 \cup Y_2$
with $\pi: Y_1 \cong P(\mathscr{O}_E(-1, -1)\oplus \mathscr{O}) \to
E$ a $P^1$ bundle and $Y_2 \cong \tilde E$. Denote by $E_0 = E =
Y_1 \cap Y_2$ and $E_\infty \cong E$ the zero and infinity divisors
of $Y_1$ respectively.

The idea is to analyze the degeneration formula for
\begin{equation*}
\langle a_1, \cdots, a_n, \tau_{\mu_1 - 1}\s_1, \cdots,
\tau_{\mu_\rho - 1} \s_\rho \rangle^{\tilde E}_{g, \beta_S, d_2}
\end{equation*}
since formally it sums over the same curve classes $\beta$ as
those in $\langle a_1, \cdots, a_n \mid \s, \mu\rangle_{g, \beta_S}$
such that
\begin{align*}
 &\sum_{j = 1}^n \deg a_j + |\mu| - \rho + \sum_{j = 1}^\rho
(\deg \s_j + 1) \\= &(c_1(\tilde E).\beta) + (\dim \tilde E - 3)(1
- g) + n + \rho.
\end{align*}

As in the proof of Proposition \ref{reduction}, we consider the
generating series of invariants with possibly disconnected domain
curves while keeping the total contact order $d_2 = |\mu|$. Then we
degenerate the series according to the contact order.

\medskip

We first analyze the splitting of curve classes. Under $N_1(\tilde
E) = i_* N_1(Z) \oplus \mathbb{Z}\gamma$, $\beta = \beta_Z + d_2
\gamma$ may be split into
\begin{equation*}
\beta^1 \in NE(Y_1) \subset NE(E) \oplus \mathbb{Z}\bar \gamma,
\quad \beta^2 \in NE(Y_2) \equiv NE(\tilde E),
\end{equation*}
such that
\begin{equation*}
(\beta^1, \beta^2) = (\beta_E^1 + c\bar\gamma, \beta_Z^2 +
e\gamma)
\end{equation*}
is subject to the condition $\phi_*\beta^1 + p_*\beta^2 = \beta$,
i.e.
\begin{equation*}
\bar\phi_*\beta_E^1 + \beta_Z^2 = \beta_Z, \quad c = d_2 \ge 0,
\end{equation*}
and the contact order relation
\begin{equation*}
e = (E.\beta^2)^{\tilde E} = (E.\beta^1)^{Y_1} = c +
(E.\beta_E^1)^{Y_1} = d_2 - (E.\beta_E^1)^{\tilde E}.
\end{equation*}

As an effective class in $E$, $\beta_E^1$ is also effective in
$\tilde E$, hence $\beta_E^1 = \zeta + m\gamma$ with $\zeta \in
NE(Z)$ and $m \in \mathbb{Z}$. It is clear that $\zeta =
\bar\phi_*\beta_E^1$ and $m = (E.\beta_E^1)^{\tilde E}$. It should
be noticed that
\begin{equation*}
e = d_2 - m
\end{equation*}
is not necessarily smaller than $d_2$ since $m$ maybe negative.
This causes no trouble since we always have that
\begin{equation*}
\beta - \beta^2 = (\beta_Z + d_2 \gamma) - (\beta_Z^2 + e\gamma) =
\bar\phi_* \beta_E^1 + m\gamma = \beta_E^1 \ge 0.
\end{equation*}
The equality holds if and only if $\beta_E^1 = 0$ and in that case
we arrive at fiber class integrals on $(Y_1, E)$ with $\beta^1 =
d_2 \bar\gamma$.

In fact, more is true. It is automatic that $[\beta] > [\beta^2]$
under the curve class splitting. The equality $[\beta] = [\beta^2]$
occurs if and only if $\beta^1_E$ consists of extremal rays
$d_1\ell$. But extremal rays must stay inside $Z$, hence we again
conclude that $\beta^1_E = 0$ and get fiber integrals on $(Y_1,
E)$. No summation over extremal rays is needed for these integrals.

\medskip

Next we analyze the splitting of cohomology insertions. 

It is sufficient to consider $(\s_1, \ldots, \s_\rho) = \e_I = (\e_{i_1},
\ldots, \e_{i_\rho})$. Since $\s_i|_Z = 0$, one may choose the
cohomology lifting $\s_i(0) = (\iota_{1*}\, \s_i, 0)$. This ensures
that insertions of the form $\tau_k\,\s$ must go to the $Y_1$ side
in the degeneration formula.

For a general cohomology insertion $\alpha \in H(\tilde E)$, by
Lemma \ref{switch}, the lifting can be chosen to be $\alpha(0) =
(a, \alpha)$ for some $a$. From $\alpha(0) = (a, \alpha)$ and
$\T\alpha(0) = (a', \T\alpha)$, Lemma \ref{coh-red} implies that $a
= a'$.

\medskip

As before the relative invariants on $(Y_1, E)$ can be regarded as
constants under $\T$. Then
\begin{align*}
&\langle a_1, \cdots, a_n, \tau_{\mu_1 - 1}\e_{i_1},
\cdots, \tau_{\mu_\rho - 1} \e_{i_\rho} \rangle^{\bullet\tilde
E}_{g, \beta_S, d_2} = \sum_{\mu'} m(\mu') \times \\
&\quad \sum_{I'} \langle \tau_{\mu_1 - 1} \e_{i_1}, \cdots,
\tau_{\mu_\rho - 1} \e_{i_\rho} \mid \e^{I'}, \mu'
\rangle^{\bullet (Y_1, E)}_{0, 0} \langle a_1, \cdots,
a_n \mid \e_{I'}, \mu' \rangle^{(\tilde E, E)}_{g, \beta_S} +
R,
\end{align*}
where the main terms contain invariants whose $(\tilde E, E)$
components admit the highest order with respect to the first four
induction parameters
\begin{equation*}
(g, \beta_S, |\mu| = d_2, n).
\end{equation*}
In fact, the potentially highest order term $\langle a_1,
\cdots, a_n \mid \e_I, \mu\rangle_{g, \beta_S}^{(\tilde E,
E)}$ occurs by the dimension count at the beginning of the proof.
Yet it is not clear a priori whether it is also the highest one in
$\rho$.

For the the remaining terms $R$, a term is in it if each connected
component of its relative invariants on $(\tilde E, E)$ has either
smaller genus or has $\beta^2_S$ strictly smaller than $\beta_S$ or
has smaller contact order or has fewer insertions than $n$. Notice
that disconnected invariants on $(\tilde E, E)$ must lie in $R$.

For the main terms, by the genus constraint and the fact that the
invariants on $(\tilde E, E)$ are connected, the invariants on
$(Y_1, E)$ must be of genus zero and the connected components are
indexed by the contact points. Also each connected invariant
contains fiber integrals with total fiber class $\beta^1 = d_2
\bar\gamma$.

To get constraints about $(\e_{I'}, \mu')$ and $\rho'$ on the main
terms, we recall the dimension count on $\tilde E$ and $(\tilde E,
E)$. Let $D = (c_1(\tilde E).\beta) + (\dim \tilde E - 3)(1 - g)$.
For the absolute invariant on $\tilde E$,
\begin{equation*}
\sum_{j = 1}^n \deg a_j + |\mu| - \rho + \sum_{j = 1}^\rho
(\deg \e_{i_j} + 1) = D + n + \rho,
\end{equation*}
while on $(\tilde E, E)$ (notice that now $(c_1(\tilde E).\beta^2)
= (c_1(\tilde E).\beta$)),
\begin{equation*}
\sum_{j = 1}^n \deg a_j + \sum_{j = 1}^{\rho'} \deg \e_{i'_j}
= D + n + \rho' - |\mu'|.
\end{equation*}
Hence $(\e_I, \mu)$ occurs in $(\e_{I'}, \mu')$'s and in
particular, $R$ is $\T$-invariant by induction. Moreover,
\begin{equation*}
\deg \e_I - \deg \e_{I'} = \rho - \rho'.
\end{equation*}

We will show that the highest order term in the main terms, with
respect to all five parameters, consists of the single one
\begin{equation*}
C(\mu)\langle a_1, \cdots, a_n \mid \e_I,
\mu\rangle^{(\tilde E, E)}_{g, \beta_S}
\end{equation*}
with $C(\mu) \ne 0$.

For any $(\e_{I'}, \mu')$ in the main terms, consider the splitting
of weighted partitions
\begin{equation*}
(\e_I, \mu) = \coprod_{k = 1}^{\rho'} (\e_{I^k}, \mu^k)
\end{equation*}
according to the connected components of the relative moduli of
$(Y_1, E)$, which are indexed by the contact points of $\mu'$.

Since fiber class relative invariants on $P^1$ bundles over $E$ can
be computed by pairing cohomology classes in $E$ with certain
Gromov--Witten invariants in the fiber $P^1$ (c.f.\ \cite{MP},
\S1.2), we must have $\deg \e_{I^k} + \deg \e^{i'_k} \le \dim E$ to
get non-trivial invariants. That is
\begin{equation*}
\deg \e_{I^k} = \sum_{j} \deg \e_{i^k_j} \le \dim E - \deg
\e^{i'_k} \equiv \deg \e_{i'_k}
\end{equation*}
for each $k$. In particular, $\deg \e_I \le \deg \e_{I'}$, hence
also $\rho \le \rho'$.

The case $\rho < \rho'$ is handled by the induction hypothesis, so
we assume that $\rho = \rho'$ and then $\deg \e_{I^k} = \deg
\e_{i'_k}$ for each $k = 1, \ldots, \rho'$. In particular $I^k \ne
\emptyset$ for each $k$. This implies that $I^k$ consists of a
single element. By reordering we may assume that $I^k = \{i_k\}$
and $(\e_{I^k}, \mu^k)= \{(\e_{i_k}, \mu_k)\}$.

Since the relative invariants on $Y_1$ contain genus zero fiber
integrals, the virtual dimension for each $k$ (connected component
of the relative virtual moduli) is
\begin{align*}
&2\mu'_k + (\dim Y_1 - 3) + 1 + (1 - \mu'_k) \\
&\qquad = (\mu_k - 1) + (\deg \e_{i_k} + 1) + (\dim E - \deg
\e_{i'_k}).
\end{align*}
Together with $\deg \e_{i_k} = \deg \e_{i'_k}$, this implies that
\begin{equation*}
\mu'_k = \mu_k, \quad k = 1, \ldots, \rho.
\end{equation*}

From the fiber class invariants consideration and
\begin{equation*}
\deg \e_{i_k} + \deg \e^{i'_k} = \dim E,
\end{equation*}
$\e_{i_k}$ and $\e^{i'_k}$ must be Poincar\'e dual to get
non-trivial integral over $E$. That is, $\e_{i'_k} = \e_{i_k}$ for
all $k$ and $(\e_{I'}, \mu') = (\e_I, \mu)$. This gives the term we
expect where $C(\mu)$ is a product of nontrivial fiber class
invariants
\begin{equation*}
\prod_{k = 1}^\rho \left(\langle \tau_{\mu_k - 1} \e_{i_k} \mid
\e^{i_k}, \mu_k\rangle_{0,\,\mu_k\bar\gamma}^{(Y_1, E)} \, q^{\mu_k
\bar\gamma} \right) = c_\mu q^{d_2\bar\gamma}
\end{equation*}
with $c_\mu \ne 0$.

In order to compare with the series $\langle a_1, \cdots,
a_n, \tau_{\mu_1 - 1}\e_{i_1}, \cdots, \tau_{\mu_\rho - 1}
\e_{i_\rho} \rangle^{\tilde E}_{g, \beta_S, d_2}$, which satisfies
the functional equation under $\T$ by assumption, we need only to
match the formal variables involved. Under $\phi: Y_1 \to \tilde E$
we set $q^{\bar\gamma} \mapsto q^\gamma$ and under $p: Y_2 \cong
\tilde E \to \tilde E$ we set $q^{\gamma} \mapsto q^0 = 1$.
Similarly we identify formal variables in the $\tilde E'$ side. It
is clear that these identifications commute with $\T$. Hence
\begin{equation*}
\T\langle a_1, \cdots, a_n \mid \e_I, \mu\rangle_{g,
\beta_S}^{\tilde E} \cong \langle \T a_1, \cdots, \T a_n
\mid \e_I, \mu\rangle_{g, \beta_S}^{\tilde E'},
\end{equation*}
and the proof of Proposition~\ref{p:4.8} is complete.
\end{proof}

\section{Reconstructions on local models} \label{reconstuction}

In this section, $X$ and $X'$ are the projective local models (double projective bundles over $S$) of the flop
\begin{equation*}
f: X = {\tilde E} = P_Z(N_{Z/X}\oplus\mathscr{O}) \dasharrow
X' = {\tilde E}' = P_{Z'}(N_{Z'/X'}\oplus\mathscr{O}).
\end{equation*}

Since we consider only genus zero invariants for the discussion on
big quantum rings, the subscript on genus will be omitted. One special feature for genus zero GW theory is that there exists several reconstruction theorems which allow us to deal with only some initial GW invariants.

By Leray--Hirsch,
$$
H(X) = H(S)[h, \xi]/(f_F(h), f_{N \oplus \mathscr{O}}(\xi)).
$$
So every $a \in H(X)$ admits a canonical presentation $a = \bart h^i \xi ^j$ with $0 \le i \le r$, $0 \le j \le r + 1$ and $\bart\in H(S)$. (In this case $\T a = \bart (\T h)^i (\T \xi)^j = \bart (\xi' - h')^i \xi'^j$ for $i \le r$ and for any $j$.) We abuse notations by writing $\xi | a$ if $j \ge 1$.

\begin{definition} [($f$-special invariants)]
An insertion $\tau_k a$ is called \emph{special} if $k \ne 0$ implies that $\xi | a$. A (possibly) descendent invariant is $f$-special it is not extremal (i.e.~$(\beta_S , d_2) \ne (0, 0)$) and if all of its insertions are special. An $f$-special invariant is of type I if $\xi$ divides some insertion, otherwise it is called of type II.
\end{definition}

\subsection{Topological recursion relation and divisor axiom}

\begin{theorem} \label{comp-red-local}
The $\T$-invariance for descendent invariants of $f$-special type is equivalent to the $\T$-invariance of big quantum rings.
\end{theorem}

\begin{proof}
We only need to prove ``$\Leftarrow$'':

Consider the generating series $\langle \tau_{k_1} a_1, \cdots, \tau_{k_n} a_n\rangle_{\beta_S, d_2}$ of $f$-special type with $(\beta_S, d_2) \ne (0, 0)$. Let $k = \sum_{i} k_i$ be the total descendent degree. We will prove the theorem by induction on $k$.

If $k = 0$, we may assume that $n \ge 3$ by adding divisors $\xi$ or $D \in H^2(S)$  into the insertions. Since $(\xi.\ell) = 0 = (D.\ell)$, this only affects the series by a nonzero constant, hence the $\T$-invariance reduces to the case of big quantum ring.

Now let $k > 0$. Without loss of generality we assume that $k_1 \ge 1$. By induction the results holds for strictly smaller descendent degree and for any $n \ge 1$.

We first treat the case $n \ge 3$. By the \emph{topological recursion relation}
$$
\psi_1 = [D_{1|2, 3}]^{virt},
$$
we get
\begin{align*}
&\langle \tau_{k_1} a_1, \cdots, \tau_{k_n} a_n\rangle_{\beta_S, d_2} \\
&\qquad = \sum_{\mu}  \langle \tau_{k_1 - 1} a_1, \cdots, T_\mu\rangle_{\beta_S', d_2'} \langle T^\mu, \tau_{k_2} a_2, \tau_{k_3} a_3, \cdots \rangle_{\beta_S'', d_2''},
\end{align*}
where the sum is over all splitting of curve classes such that $(\beta_S', d_2') + (\beta_S'', d_2'') = (\beta_S, d_2)$.

Notice that on the RHS, the case $(\beta_S', d_2') = (0, 0)$ is excluded since $\xi|a_1$ and it will lead to trivial invariants. The $(\beta_S', d_2')$ series is then $\T$-invariant since it has strictly smaller descendent order $k_1 - 1 < k$. (Recall that on the $X'$ side we may choose $\T T_\mu$ and $\T T^\mu $ for the splitting since $\T$ preserves the Poincar\'e pairing.)

The $(\beta_S'', d_2'')$ series is also $\T$-invariant: It has strictly smaller descendent degree and it has at least $3$ insertions. So even if $(\beta_S'', d_2'') = (0, 0)$ we still get the $\T$-invariance.

The case $n = 1$ can be reduced to the case $n = 2$ by the divisor equation for descendant invariants. Namely let $b$ be a divisor coming from the base $S$ or $\xi$ such that $b.(\beta_S + d_2\gamma) \ne 0$. Then $(b.\beta) \ne 0$ is independent of $d$ and
$$
\langle b, \tau_k a\rangle_{\beta_S, d_2} = (b.\beta)\langle \tau_k a \rangle_{\beta_S, d_2} + \langle \tau_{k - 1} ab\rangle_{\beta_S, d_2}.
$$

The case $n = 2$ can be similarly reduced to the case $n = 3$. If there is only one descendent insertion, say $\langle a_1, \tau_k a_2 \rangle_{\beta_S, d_2}$, then
$$
\langle b, a_1, \tau_k a_2 \rangle_{\beta_S, d_2} = (b.\beta)\langle a_1, \tau_k a_2 \rangle_{\beta_S, d_2} + \langle a_1, \tau_{k - 1} a_2 b\rangle_{\beta_S, d_2}.
$$
If there are two descendent insertions, say $\langle \tau_{l} a_1, \tau_{k - l} a_2 \rangle_{\beta_S, d_2}$, then
\begin{align*}
\langle b, \tau_{l} a_1, \tau_{k - l} a_2 \rangle_{\beta_S, d_2} &= (b.\beta) \langle \tau_{l} a_1, \tau_{k - l} a_2 \rangle_{\beta_S, d_2} \\
&\quad + \langle \tau_{l - 1} a_1 b, \tau_{k - l} a_2 \rangle_{\beta_S, d_2} + \langle \tau_{l} a_1, \tau_{k - l - 1} a_2 b \rangle_{\beta_S, d_2}.
\end{align*}
All the other series are either 3-point functions or have descendent degree drops by one. Thus by induction the proof is complete.
\end{proof}

\subsection{Divisorial reconstruction and quasi-linearity}

Theorem \ref{comp-red-local} reduces the analytic continuation problem to the local models completely. However, in the actual determination of GW invariants (as will see in later sections), another natural set of initial GW invariants are those with at most one descendent insertion. This suggests another reconstruction procedure.

\begin{definition} [(Quasi-linearity)]
We say that the flop $f$ is quasi-linear if for every special insertion $\alpha \in H(X)\cup\tau_{\bullet}H(E)$, $\bart_i \in H(S)$ and $(\beta_S, d_2) \ne (0, 0)$, we have
$$
\T\langle \bart_1, \cdots, \bart_{n - 1}, \alpha\rangle^X_{\beta_S, d_2} \cong
\langle \bart_1, \cdots, \bart_{n - 1}, \T\alpha\rangle^{X'}_{\beta_S, d_2}.
$$
We call invariants of the above type (with only one insertion not from the base) \emph{elementary}. Quasi-linearity is the $\T$-invariance for elementary $f$-special invariants.\end{definition}

 Notice that the similar statement for descendent invariants, even for simple flops, is generally wrong if $\alpha = \tau_k a$ with $k > 0$ but $a \not\in H(E)$ (c.f.~\cite{LLW}).

\begin{theorem} \label{p:5.1}
Suppose that $f$ is quasi-linear. Then all descendent invariants of $f$-special type are $\T$-invariant. Namely for $\alpha =
(\alpha_1,\ldots,\alpha_n)$ ($n \ge 1$) with $\alpha_i \in
H(X)\cup\tau_{\bullet}H(E)$ and for $(\beta_S, d_2) \ne (0, 0)$, we have
$$
\T\langle \alpha\rangle^{X}_{\beta_S, d_2} \cong \langle \T\alpha
\rangle^{X'}_{\beta_S, d_2}.
$$

More precisely, any series of $f$-special type can be reconstructed, in an $\T$-compatible manner, from the extremal functions with $n \ge 3$ points and  elementary $f$-special series.
\end{theorem}

We will prove the reconstruction by induction on $(\beta_S, d_2) \in W$, and then on $m$ which is the number of insertions not coming from base classes. This is based on the following observations: \smallskip

(1) Under \emph{divisorial reconstruction}: $\psi_i + \psi_j = [D_{i|j}]^{virt}$, and for $L \in {\rm Pic}(X)$,
\begin{equation} \label{div-rel-2}
e_i^* L = e_j^*L + (\beta.L)\psi_j - \sum_{\beta_1 + \beta_2 = \beta} (\beta_1.L) [D_{i \beta_1| j \beta_2}]^{virt}
\end{equation}
(\cite{LP}, c.f.~also \cite{LLW}), the degree $\beta$ is either preserved
or split into effective classes $\beta = \beta_1 + \beta_2$.

(2) When summing over $\beta \in (d_2 \gamma +
\bar\psi^*\beta_S.H_r + \mathbb{Z} \ell) \cap NE(X)$, the splitting
terms can usually be written as the product of two generating
series with no more marked points in a manner which will be clear
in each context during the proof.

\smallskip

We also need to comment on the excluded cases $(\beta_S, d_2) = (0, 0)$:\smallskip

(3) Let $\alpha_i = \tau_{k_i} a_i$. If $k = \sum k_i \ne 0$, say $\xi|a_1$, then the extremal invariants survive only for the case $\beta = 0$. Since $\Mbar_{0, n}(X, 0)
\cong \Mbar_{0, n} \times X$, we have
\begin{equation} \label{degree=0}
\langle \tau_{k_1} a_1, \cdots, \tau_{k_n} a_{n} \rangle_{n, \beta = 0} =
\int_{\Mbar_{0, n}} \psi_1^{k} \times \int_X a_1\cdots a_n.
\end{equation}
It is non-trivial only if $k = \dim \Mbar_{0, n} = n - 3$, and then
$$
\int_X a_1 \cdots a_n =
\int_{X'} \T a_1 \cdots \T a_n
$$
since the flop $f$ restricts to an isomorphism on $E$.

(4) For extremal invariants with $k = 0$, since $\xi|_Z = 0$
and the extremal curves will always stay in $Z$, we get trivial
invariant if one of the insertions involves $\xi$. Hence by
Theorem~\ref{big-q} the statement in the theorem still holds in
this initial case except for the 2-point invariants $\langle
\bart_1h^r, \bart_2h^r\rangle$. By the divisor axiom
$$
\delta_h \langle \bart_1h^r, \bart_2h^r\rangle = \langle h, \bart_1h^r,
\bart_2h^r\rangle_+,
$$
the 2-point invariants will satisfy the $\T$-invariance functional equation up to
analytic continuation only after incorporated with classical
defect. Thus we may base our induction on $(\beta_S, d_2) = (0,0)$
with special care taken to handle this case.\smallskip

\begin{proof}
Let $(\beta_S, d_2) \ne (0, 0)$. If $m = 1$ then we are done, so let $m \ge 2$ . \smallskip

Step 1. First we handle the type I case, i.e.~ with the
appearance of $\xi$ in some $\alpha_i$.

By reordering we may assume that $\alpha_n = \tau_s \xi a$, $s \ge
0$. Write
$$
\alpha_1 = \bart_1\tau_k h^l\xi^j.
$$
We will reduce $m$ by moving divisors in $\alpha_1$ into
$\alpha_n$ in the order of $\psi$, $h$ and $\xi$. This process is compatible with $\T$ since $\T a.\T \xi = \T(a.\xi)$.

For $\psi$, we use the equation
$$
\psi_1 = -\psi_n + [D_{1|n}]^{virt}.
$$
If $k \ge 1$ then $j \ne 0$ and we get
\begin{align*}
\langle \bart_1\tau_k h^l\xi^j, \cdots, \tau_s \xi a\rangle_{\beta_S,
d_2} &= -\langle \bart_1\tau_{k - 1} h^l\xi^j, \cdots, \tau_{s + 1}
\xi a\rangle_{\beta_S, d_2} \\
&\quad + \sum_{\mu} \langle \bart_1\tau_{k - 1} h^l\xi^j, \cdots,
T_{\mu}\rangle_{\beta_S', d_2'} \langle T^{\mu}, \cdots, \tau_s \xi
a\rangle_{\beta_S'', d_2''}.
\end{align*}
For each $i$, if one of $(\beta_S', d_2')$ and $(\beta_S'',
d_2'')$ is $(0,0)$ then since both terms contain $\xi$ the
splitting term must vanish. So we may assume that
$$
(\beta_S', d_2') < (\beta_S, d_2) \quad \mbox{and} \quad
(\beta_S'', d_2'') < (\beta_S, d_2)
$$
and these terms are done by the induction hypothesis. (By performing
this procedure to $\alpha_1, \ldots, \alpha_{n - 1}$ we may assume
that the only descendent insertion is $\alpha_n$.)

For $h$, if $l \ge 1$ we use the divisor relation (\ref{div-rel-2}) for $L = h$ to get
\begin{align*}
&\langle \bart_1h^l\xi^j, \cdots, \tau_s \xi a\rangle_{\beta_S, d_2}
\\
&\quad = \langle \bart_1h^{l - 1}\xi^j, \cdots, \tau_s \xi
ah\rangle_{\beta_S, d_2} + \delta_h\langle \bart_1h^{l - 1}\xi^j,
\cdots, \tau_{s + 1} \xi a\rangle_{\beta_S, d_2} \\
&\qquad - \sum_{\mu} \delta_h\langle \bart_1h^{l - 1}\xi^j, \cdots,
T_{\mu}\rangle_{\beta_S', d_2'} \langle T^{\mu}, \cdots, \tau_s \xi
a\rangle_{\beta_S'', d_2''}.
\end{align*}

The only cases for the splitting term to have one factor with the
same $(\beta_S, d_2)$ and $m$ are of the form (denote by $\bart_*$ some set of insertions $\alpha_j \in H(S)$)
$$
\delta_h\langle \bart_1h^{l - 1}\xi^j, \bart_*, T_{\mu} \rangle_{0,0} \langle T^{\mu},
\cdots, \tau_s \xi a \rangle_{\beta_S,
d_2},
$$
where the LHS has $n'$ points, or
$$\delta_h\langle \bart_1 h^{l - 1} \xi^j,
\cdots,  T_{\mu}\rangle_{\beta_S, d_2} \langle T^{\mu}, \bart_*,
\tau_s \xi a\rangle_{0,0}.
$$
But $l - 1 < r$ forces the former LHS
invariants to vanish: For $j \ne 0$ this is trivial. For $j = 0$,  the codimension (c.f.~\S \ref{q-correction})
\begin{equation} \label{codim}
\mu = |h| - (2r + 1 + n' - 3) < 2r - 2r = 0.
\end{equation}
The latter RHS invariants also vanish since they contain $\xi$.

If $j = 0$, the case $(\beta_S', d_2') = (0, 0)$ may still support nontrivial
invariants with 3 or more points. In that case $m$ decreases in the RHS. For  the other terms, the only possible appearance of type II invariants (i.e.~without $\xi$ insertion) is
\begin{equation} \label{new-typeII}
\delta_h\langle \bart_1h^{l - 1}, \cdots, T_{\mu}\rangle_{\beta_S', d_2'} = \langle h, \bart_1h^{l - 1}, \cdots, T_{\mu}\rangle_{\beta_S', d_2'},
\end{equation}
where $j = 0$, which has at least 3 points and $(0, 0) < (\beta_S', d_2') < (\beta_S, d_2)$.

For $\xi$, the argument is entirely similar. For $j \ge 1$, the
divisor relation says that
\begin{align*}
 &\langle \bart_1\xi^j, \cdots, \tau_s \xi a\rangle_{\beta_S, d_2} \\
 &\quad = \langle \bart_1\xi^{j - 1}, \cdots, \tau_s \xi^2 a \rangle_{\beta_S,
d_2} + \delta_\xi \langle \bart_1\xi^{j - 1}, \cdots, \tau_{s + 1} \xi
a\rangle_{\beta_S, d_2}\\
 &\qquad - \sum_{\mu} \delta_\xi\langle \bart_1\xi^{j - 1}, \cdots,
T_{\mu}\rangle_{\beta_S', d_2'} \langle T^{\mu}, \cdots, \tau_s \xi
a\rangle_{\beta_S'', d_2''}.
\end{align*}
We then have $(\beta_S', d_2') < (\beta_S, d_2)$ and $(\beta_S'',
d_2'') < (\beta_S, d_2)$ as before. Notice that only type I invariants appear in the reduction.
\smallskip

Step 2. Next we deal with the type II case: $\alpha_i = \bart_i h^{l_i}$, $1 \le i \le n$. In case $\beta_S=0$, we can add one $\xi$ into the insertions and then go back to Step 1. From (\ref{new-typeII}), $(\beta_S,d_2)$ will be getting smaller when the possible type II invariants appear again, so it is done by induction. Thus we can allow $\beta_S\ne 0$ here. By adding base divisors into the insertions we may always assume that $n \ge 3$.

We can not apply (\ref{div-rel-2}) to move divisors since it will produce non $f$-special invariants. Instead, since $n \ge 3$ we may apply (\ref{div-rel-1}), the descendent-free form of the divisor relation, as we have used in the proof of Theorem \ref{big-q}.

Suppose that $l_1 > 0$ and $l_2 > 0$ and we move $h$ from $\alpha_1$ to $\alpha_2$. We run induction on $l_1$. Namely we assume the $\T$-invariant reduction holds for $\alpha_1 = \bart_1 h^j$ with $j \le l_1 - 1$. The initial case $j = 0$ holds since $m$ drops by 1. Then
\begin{align*}
&\langle \bart_1 h^{l_1}, \bart_2 h^{l_2}, \alpha_3, \cdots \rangle_{\beta_S, d_2} \\
&\quad =\langle \bart_1 h^{l_1 - 1}, \bart_2 h^{l_2 + 1}, \alpha_3, \cdots \rangle_{\beta_S, d_2} \\
 &\qquad + \sum_{\mu} \langle \bart_1 h^{l_1 - 1}, \alpha_3, \cdots, T_{\mu} \rangle_{\beta_S', d_2'} \delta_h \langle \bart_2 h^{l_2}, \cdots, T^{\mu} \rangle_{\beta_S'', d_2''} \\
 &\quad\qquad - \delta_h \langle \bart_1 h^{l_1 - 1}, \cdots, T_{\mu} \rangle_{\beta_S', d_2'} \langle \bart_2 h^{l_2}, \alpha_3, \cdots, T^{\mu} \rangle_{\beta_S'', d_2''}.
\end{align*}

If $l_2 \le r - 1$, the processes on $X$ and $X'$ are clearly $\T$-compatible and the splitting terms are all handled by induction. Indeed, if $(\beta_S', d_2') = (\beta_S, d_2)$ and $m' = m$ then $(\beta_S'', d_2'') = (0, 0)$ which gives an extremal function with $m'' \le 2$. The analogous codimension condition as in (\ref{codim}) forces the term to vanish. Similar consideration applies to the case $(\beta_S'', d_2'') = (\beta_S, d_2)$ as well.

If $l_2 = r$, the first term is no longer $\T$-compatible. The topological defect of the second insertion is given by Lemma \ref{defect-2}: $\T (h^{r + 1}) - (\T h)^{r + 1} = (-1)^{r + 1} \T\Theta_{r + 1}$, where $\Theta_{r + 1}$ is the dual class of ${\rm pt}.h^r \xi^0$. Meanwhile, the splitting terms also contain one term not of lower order in $(\beta_S, d_2)$ and $m$. By the codimension consideration as in (\ref{codim}), we have $T^{\mu} = \hatt \bart_2 h^r$
and the term is given by
$$
\langle \bart_1 h^{l_1 - 1}, \alpha_3, \cdots, \alpha_n,  \bart_2 \Theta_{r + 1} \rangle_{\beta_S, d_2} \delta_h \langle \bart_2 h^r, \hatt \bart_2  h^r \rangle_{0, 0}.
$$
Comparing with its corresponding term on $X'$
$$
\langle \bart_1 \T h^{l_1 - 1}, \T \alpha_3, \cdots, \T \alpha_n,  \bart_2 \T \Theta_{r + 1} \rangle_{\beta_S, d_2} \delta_{\T h} \langle \bart_2 \T h^r,
\hatt \bart_2  \T h^r \rangle_{0, 0}
$$
and using the induction, we get the difference to be
\begin{align*}
&-\langle \bart_1 \T h^{l_1 - 1}, \T \alpha_3, \cdots, \T \alpha_n,  \bart_2 \T \Theta_{r + 1} \rangle_{\beta_S, d_2} \times (-1)^{r + 1}\\
&= -\langle \bart_1 \T h^{l_1 - 1}, \bart_2 \T (h^{r + 1}), \cdots \rangle_{\beta_S, d_2} + \langle \bart_1 \T h^{l_1 - 1}, \bart_2 (\T h)^{r + 1}, \cdots \rangle_{\beta_S, d_2}.
\end{align*}
This cancels the defect of the non $\T$-compatible terms.

Thus the whole reduction is $\T$-invariant and the proof is complete.
\end{proof}

\subsection{WDVV equations}

We may strengthen Theorem \ref{p:5.1} to
\begin{theorem} \label{QL-I}
If the quasi-linearity holds for elementary type I series
$$
\langle \bart_1, \cdots, \bart_{n - 1}, \tau_k a \xi \rangle,
$$
then the $\T$-invariance holds for all series of $f$-special type.
\end{theorem}

The significance of this reduction will become clear after we introduce the practical method to calculate GW invariants. The proof is based on

\begin{proposition} \label{II-to-I}
Any type II series over $(\beta_S, d_2)$ can be transformed into sum of products of (1)
type I series over $(\beta_S', d_2') \le (\beta_S, d_2)$, (2) type II series over $\beta_S' < \beta_S$, and (3) extremal functions. Also, the processes can be done in a $\T$-compatible manner.
\end{proposition}

Indeed, with Proposition \ref{II-to-I}, Theorem \ref{QL-I} then follows from the proof of Theorem \ref{p:5.1}: Simply replace Step 2 by the proposition and run the induction. All type II special series eventually disappear. (\emph{Degenerate} type II series with $(\beta_S, d_2) = (0, 0)$ are simply extremal functions.)\medskip

The remaining of this subsection is devoted to the proof of Proposition \ref{II-to-I}. Notice that if $d_2 \ne 0$ then this is trivial: By the divisor axiom,
$$
\langle a_1, \cdots, a_n \rangle_{\beta_S, d_2} =
\langle a_1, \cdots, a_n, \xi\rangle_{\beta_S, d_2}/d_2.
$$

Thus we consider $\langle a_1,
\cdots, a_{n-1}, \bart_i h^j\rangle_{\beta_S, 0}$ with $a_1, \ldots, a_{n - 1} \in
H(Z)$.

Let $\{\tb_i\}$ be a basis for $H(S)$ and
$\{\hatt \tb_i\}$ be its dual basis.
We start with the case of three-point functions $\langle a, b, \tb_i
h^j \rangle_{\beta_S, 0}$ for any $a, b \in H(Z)$. This certainly
includes also the one-point and two-point cases by picking suitable
$a, b \in H^2(S)$.

For any $c, d \in H(X)$, the WDVV equations
$$
\sum_{m, n}\p_{ijm}F_0\, g^{mn}\,\p_{nkl} F_0 =
\sum_{m, n}\p_{ikm} F_0\, g^{mn}\, \p_{njl} F_0
$$
lead to the diagram
\begin{equation*}
[a\vee b \mapsto \xi
c \vee \xi d] = [a \vee \xi c \mapsto b \vee \xi d].
\end{equation*}
We apply it to split the curve classes over $(\beta_S, d_2 = 1)$
and get a linear equation
\begin{equation} \label{WDVV}
\sum_{i, j} \langle a, b, \tb_ih^j \rangle_{\beta_S, 0} \langle
\hatt \tb_i H_{r - j} \Theta_{r + 1}, \xi c, \xi d \rangle_{0, d_2} =
I_{c, d},
\end{equation}
where all terms in the LHS of WDVV with either (1) $\beta_S' < \beta_S$, (2) $d_2' \ne
0$, or (3) with basis class insertion $T_\mu = \tb_ih^j\xi^k$ ($k > 0$) from the diagonal splitting, have been
moved into the RHS. Since the original RHS of WDVV are all type I series, any series in $I_{c, d}$ over $(\beta_S', d_2')$ must satisfy $\beta_S' < \beta_S$ or $(\beta_S', d_2') = (\beta_S, 0)$.

Let $m = \sum_i h^i(S)$. We intend to form an $N \times N$
invertible system with $N = m(r + 1)$. The virtual dimension of the
second series is
$$
d_2(r + 2) + 2r + 1 + s.
$$
Thus for $d_2 = 1$, we should require $|c| + |d| = r + |\tb_i| + j$ to match the dimension.

Natural choices of $\{(c, d)\}$ are
\begin{equation}
c = c_{k, l} := \tb_k \xi^l, \quad d = h^r.
\end{equation}
The set $\{c_{k, l}\}$ is partially ordered by $|\tb_k|$ and then by $l$.

We claim that the resulting system is upper triangular with non-zero
diagonal. Indeed,
$$
\langle \hatt\tb_i H_{r - j} \Theta_{r + 1}, \tb_k \xi^{l + 1}, \xi h^r
\rangle_{0, 1} \ne 0
$$
only if $|\tb_k| + l = |\tb_i| + j$.

The key point is to use the fiber bundle structure $\Mbar_{0, n}(X,
\beta) \to S$ for $\beta = d\ell + d_2\gamma$ as in the extremal
case (where $d_2 = 0$). The fiber is given by $\Mbar_{0, n}$ of the
toric local model for the simple flop case.

Thus if $|\tb_k| > |\tb_i|$ then $|\hatt\tb_i| + |\tb_k| > s$ and the
invariant is zero. Even in the case $|\tb_k| = |\tb_i|$, and so $l =
j$, we must have $\tb_k = \tb_i$ to avoid trivial invariants. The other
cases $|\tb_k| < |\tb_i|$ belong to the strict upper triangular region
which do not affect our concern.

It remains to calculate the diagonal fiber series (sum in $d \ge
0$)
$$
 \sum_i \langle \hatt\tb_i H_{r - j} \Theta_{r + 1}, \tb_i \xi^{j + 1}, \xi h^r
\rangle_{0, 1} = \langle h^{r - j} (\xi - h)^{r + 1}, \xi^{j + 1},
\xi h^r \rangle_{d_2 = 1}^{\rm simple}.
$$
We had done a similar calculation before for the extremal case in
\cite{LLW}, Proposition 3.8. In the current case we have

\begin{lemma}
For simple flops, the fiber series in $d$ with $d_2 = 1$ are given
by
$$
\langle h^{r - j} (\xi - h)^{r + 1}, \xi^{j + 1}, \xi h^r
\rangle_{d_2 = 1} =
\begin{cases}
(-1)^j q^\ell q^\gamma, \quad 0 \le j \le r - 1;\\
(1 - (-1)^{r + 1}q^\ell) q^\gamma, \quad j = r.
\end{cases}
$$
\end{lemma}

\begin{proof}
By applying the divisor relation to move one $\xi$ class with
respect to $(i, j, k) = (2, 1, 3)$, we get (notice that $\xi(\xi -
h)^{r + 1} = 0$)
\begin{equation*}
\begin{split}
&\langle h^{r - j} (\xi - h)^{r + 1}, \xi^{j + 1}, \xi h^r \rangle_{d_2 = 1} \\
&= \sum_{\mu} \langle \xi^j, \xi h^r , T_\mu \rangle_0 \delta_\xi \langle T^\mu,
h^{r - j}(\xi - h)^{r + 1}\rangle_1 - \delta_\xi\langle \xi^j, T_\mu\rangle_1
\langle T^\mu, h^{r - j} (\xi - h)^{r + 1}, \xi h^r\rangle_0 \\
&= \langle h^{r - j}(\xi - h)^{r + 1}, \xi^{j + 1}h^r\rangle_1.
\end{split}
\end{equation*}

By another divisor relation (\ref{div-rel-2}), we can keep track on the 2-point invariants as follows:
\begin{equation*}
\begin{split}
&\langle h^{r - j}(\xi - h)^{r + 1}, \xi^{j + 1}h^r\rangle_1\\
&= \langle \psi h^{r - j}(\xi - h)^{r + 1}, \xi^{j}h^r\rangle_1 -
\sum_{\mu} \delta_{\xi}\langle \xi^jh^r, T_\mu\rangle_1\langle T^\mu,
h^{r-j}(\xi-h)^{r+1}\rangle_0\\
&= \langle \psi h^{r - j}(\xi - h)^{r + 1}, \xi^{j}h^r\rangle_1=\cdots \\
&= \langle \psi^{j+1}h^{r - j}(\xi - h)^{r + 1}, h^r\rangle_1.
\end{split}
\end{equation*}
Here we use the fact that there is no extremal invariants with any
insertion involving $\xi$ (notice that $(\xi-h)^{r+1} =
\xi(\cdots)$ since $h^{r+1}=0$).

Next we move the divisor class $h$ in $h^r$ to the left one by one:
\begin{equation*}
\begin{split}
&\langle \psi^{j+1}h^{r - j}(\xi - h)^{r + 1}, h^r\rangle_1\\
&= \langle \psi^{j+1}h^{r - j+1}(\xi - h)^{r + 1},
h^{r-1}\rangle_1 + \delta_h\langle \psi^{j+2}h^{r - j}(\xi - h)^{r
+ 1}, h^{r-1}\rangle_1\\
&\qquad - \sum_{\mu} \delta_h\langle h^{r-1}, T_\mu\rangle_0\langle T^\mu,
\psi^{j+1}h^{r - j}(\xi - h)^{r + 1}\rangle_1\\
&=\langle \psi^{j+1}(h+d\psi)h^{r - j}(\xi - h)^{r + 1},
h^{r-1}\rangle_1 = \cdots \\
&=\langle \psi^{j+1}(h+d\psi)^{r-1}h^{r - j}(\xi - h)^{r + 1},
h\rangle_1.
\end{split}
\end{equation*}
Note that $\langle h^{r-1}, T_\mu\rangle_0=0$ since the power of
$h$ is less than $r$.

Finally, the divisor axiom helps us to obtain the result:
\begin{equation*}
\begin{split}
&\langle \psi^{j+1}(h+d\psi)^{r-1}h^{r - j}(\xi - h)^{r + 1},
h\rangle_1\\
&=d\langle \psi^{j+1}(h+d\psi)^{r-1}h^{r - j}(\xi - h)^{r +
1}\rangle_1 + \langle h\psi^{j}(h+d\psi)^{r-1}h^{r - j}(\xi -
h)^{r + 1}\rangle_1\\
&=\langle \psi^{j}(h+d\psi)^{r}h^{r - j}(\xi - h)^{r +
1}\rangle_1,
\end{split}
\end{equation*}
which is the constant term in the $z$ expansion in
\begin{equation*}
\begin{split}
&\Big<\sum_{k\geq 0}\frac{\psi^k}{z^k}z^{j}(h+dz)^{r}h^{r -
j}(\xi - h)^{r + 1}\Big>_1\\
&=z^{j+2}{e_1}_*\Big(\frac{1}{z(z-\psi)}e_1^*(h+dz)^{r}h^{r -
j}(\xi - h)^{r + 1}\Big).
\end{split}
\end{equation*}

According to the same discussion of quasi-linearity in \cite{LLW},
if $d_2-d<0$ then $P_\beta$ vanishes after multiplication by $\xi$.
Here $h^{r - j}(\xi - h)^{r + 1}$ does contain at least one $\xi$.
Hence we only need to consider $d_2 \geq d$. Now $d_2 = 1$, thus $d
= 0$ or $1$.

If $d=0$, then $h^rh^{r-j}(\xi-h)^{r+1}$ is nontrivial only if
$j=r$ and in this case we get $h^r(\xi-h)^{r+1}=h^r\xi^{r+1} =
\mbox{pt}$. It is clear that the constant term of $z$ in
$$
z^{r + 2} J_\beta.\mbox{pt} =
z^{r+2}\frac{1}{(\xi-h+z)^{r+1}(\xi+z)}.{\rm pt}
$$
is equal to 1.

If $d=1$, then $J_\beta = 1/(h+z)^{r+1}(\xi+z)$. Thus
\begin{equation*}
\begin{split}
&z^{j+2}\frac{(h+z)^{r}h^{r - j}(\xi - h)^{r + 1}}{(h+z)^{r+1}(\xi
+ z)} \\
&=\frac{z^{j+2}}{z^2}\frac{h^{r - j}(\xi - h)^{r +
1}}{(1 + {h}/{z})(1 + {\xi}/{z})}\\
&=z^jh^{r-j}(\xi-h)^{r+1} \Big(1-\frac{h}{z} +
\frac{h^2}{z^2}-\cdots(-1)^j\frac{h^j}{z^j}+\cdots
\Big)\Big(1-\frac{\xi}{z} +\cdots \Big).
\end{split}
\end{equation*}
Since $\xi(\xi-h)^{r+1}=0$, the constant term is given by
$$
(-1)^jh^r(\xi-h)^{r+1}=(-1)^jh^r\xi^{r+1}=(-1)^j.
$$
The proof is complete.
\end{proof}

Now we consider $n$-point functions with $n \ge 3$. The WDVV
equation is for triple derivatives of the $g = 0$ potential
function. Let $t \in H^{>2}(X)$ be a general insertion without the
fundamental class and divisors. Then we have
\begin{equation} \label{WDVV-n}
\sum_{i, j} \langle a, b, \tb_ih^j \rangle_{\beta_S, 0}(t) \langle
\hatt \tb_i H_{r - j} \Theta_{r + 1}, \tb_k \xi^{l + 1}, \xi h^r
\rangle_{0, 1}(t) = I_{k, l}(t)
\end{equation}
where any series in $I_{c, d}$ over $(\beta_S', d_2')$ must satisfy $\beta_S' < \beta_S$ or $(\beta_S', d_2') = (\beta_S, 0)$.

By dimension counting, one more marked point increases one virtual
dimension while $t$ has Chow degree more
than one, so we find that
$$
\langle \hatt \tb_i H_{r - j} \Theta_{r + 1}, \tb_k \xi^{l + 1}, \xi h^r
\rangle_{0, 1}(t) = \langle \hatt \tb_i
 H_{r - j} \Theta_{r + 1}, \tb_k \xi^{l + 1}, \xi h^r
\rangle_{0, 1}
$$
is in fact independent of $t$ when $|\tb_i| + j = |\tb_k| + l$. The linear system (\ref{WDVV-n}) is thus $\T$-compatible by the quantum invariance of simple flop case \cite{LLW}.

In any case, if $|\tb_k| > |\tb_i|$ then the invariants are still zero.
In particular the $N \times N$ system is still upper triangular.
Moreover the diagonal entries are still given by the original 3
point (finite) series. Thus the series
$$
\langle a, b, \tb_ih^j \rangle_{\beta_S, 0}(t)
$$
are solvable in terms of the expected terms.

\end{document}